\documentclass[11pt,reqno]{amsart}
\usepackage{graphicx,amsmath,amsthm,verbatim,amssymb,eucal,enumerate}
\usepackage{bbm}
\usepackage{hyperref}
\hypersetup{colorlinks}
\pdfoutput=1

\newcommand{\eref}[1]{\eqref{eq:#1}}
\newcommand{\figref}[1]{Figure~\ref{fig:#1}}

\newcommand{\arxiv}[1]{{\tt \href{http://arxiv.org/abs/#1}{arXiv:#1}}}

\newcommand{\floor}[1]{\left\lfloor {#1} \right\rfloor}
\newcommand{\ceiling}[1]{\left\lceil {#1} \right\rceil}

\newcommand{\old}[1]{}
\newcommand{\moniker}[1]{{\em (#1)}}

\newcommand{\inner}[1]{\langle #1 \rangle} % inner product

\DeclareRobustCommand{\SkipTocEntry}[5]{}
\newtheorem{theorem}{Theorem}[section]
\newtheorem{prop}[theorem]{Proposition}

\newtheorem{lemma}[theorem]{Lemma}
\newtheorem{corollary}[theorem]{Corollary}

\newtheoremstyle{named}{}{}{\itshape}{}{\bfseries}{.}{.5em}{\thmnote{#3}#1}
\theoremstyle{named}
\newtheorem*{namedtheorem}{}

\theoremstyle{remark}

\numberwithin{counter}{section}

\theoremstyle{definition}

\def\Points{{::}}
\def\div{{\text{div}\hspace{0.5ex}}}

\def\xx{\mathbf{x}}

\def\oball{\mathrm{B}} % open ball
\def\cball{\overline{\oball}} % closed ball
\def\dball{\mathbb{B}} % discrete ball

\def\zero{\mathbf{0}} %origin
\def\green{G} %continuous Green function: mathcal uppercase G

\def\d{\,\mathrm{d}}

\def\eps{\epsilon}

\renewcommand{\phi}{\varphi}

\def\Rec{\mathrm{Rec}}
\def\Sym{\mathrm{Sym}}

\def\EE{\mathbb{E}}
\def\N{\mathbb{N}}
\def\PP{\mathbb{P}}

\def\R{\mathbb{R}}
\def\Z{\mathbb{Z}}

\def\xx{\mathbf{x}}
\def\yy{\mathbf{y}}

\def\00{\mathbf{0}}

\def\basis{\mathbf{e}}
\def\Basis{\mathcal{E}}

\newcommand{\Frame}[1]{}
\newcommand{\FIG}[1]{}

\newcommand{\stab}[1]{\hat{#1}}

\newsavebox{\smlmat}% Box to store smallmatrix content for figure caption
\savebox{\smlmat}{$\left(\begin{smallmatrix} c-a & b \\ b & c+a \end{smallmatrix}\right) \in \Gamma(\Z^2)$}

\begin{document}

\title{Laplacian growth, sandpiles and scaling limits}

\author{Lionel Levine and Yuval Peres}

\address{Lionel Levine, Department of Mathematics, Cornell University, Ithaca, NY 14853. {\tt \url{http://www.math.cornell.edu/~levine}}}

\address{Yuval Peres, Microsoft Research, Redmond, WA 98052. \\
{\tt \url{http://research.microsoft.com/en-us/um/people/peres/}}}

\thanks{LL was supported by NSF grant \href{http://www.nsf.gov/awardsearch/showAward?AWD_ID=1455272}{DMS-1455272} and a Sloan Fellowship.
}

\date{August 1, 2016}

\begin{abstract}
Laplacian growth is the study of interfaces that move in proportion to harmonic measure. Physically, it arises in fluid flow and electrical problems involving a moving boundary. We survey progress over the last decade on discrete models of (internal) Laplacian growth, including the abelian sandpile, internal DLA, rotor aggregation, and the scaling limits of these models on the lattice $\eps \Z^d$ as the mesh size $\eps$ goes to zero.  These models provide a window into the tools of discrete potential theory: harmonic functions, martingales, obstacle problems, quadrature domains, Green functions, smoothing.  We also present one new result: rotor aggregation in $\Z^d$ has $O(\log r)$ fluctuations around a Euclidean ball, improving a previous power-law bound. We highlight several open questions, including whether these fluctuations are $O(1)$.
\end{abstract}

\maketitle

\section{The abelian sandpile model}
%%% change s to \sigma in this section

Start with $n$ particles at the origin in the square grid $\Z^2$, and let them spread out according to the following rule: whenever any site in $\Z^2$ has $4$ or more particles, it gives one particle to each of its $4$ nearest neighbors (North, East, South and West). The final configuration of particles does not depend on the order in which these moves are performed (which explains the term ``abelian''; see Lemma~\ref{l.subseq} below).

\begin{figure}[h]
\centering
\begin{tabular}{cc}
\includegraphics[height=.48\textwidth]{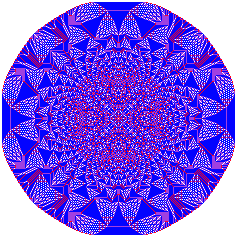} &
\includegraphics[height=.48\textwidth]{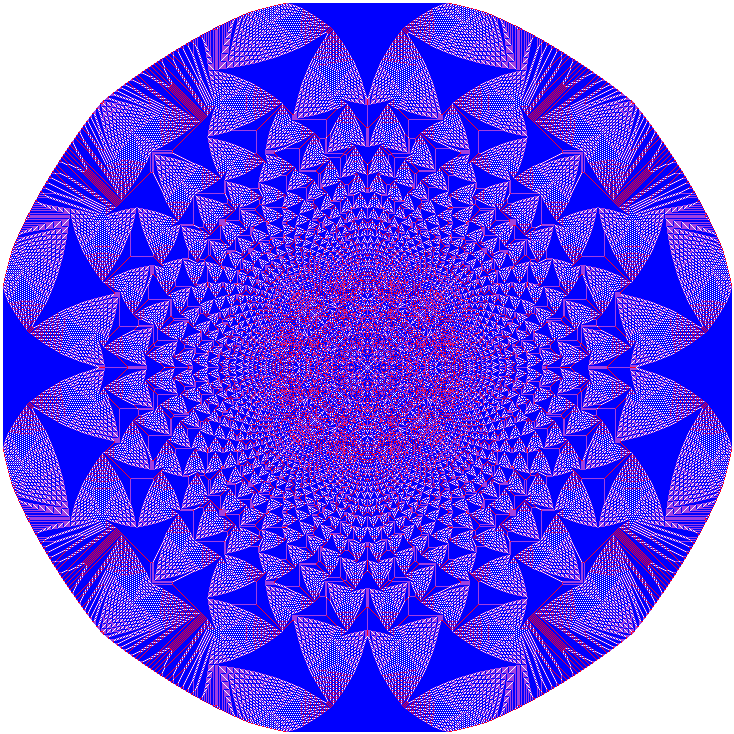} \\
$n=10^5$ & $n=10^6$
\end{tabular}
\caption{Sandpiles in $\Z^2$ formed by stabilizing $10^5$ and $10^6$ particles at the origin. Each pixel is colored according to the number of sand grains that stabilize there (white~$0$, red~$1$, purple~$2$, blue~$3$). The two images have been scaled to have the same diameter.}
\label{f.singlesource}
\end{figure}

This model was invented in 1987 by the physicists Bak, Tang and Wiesenfeld \cite{BTW87}. While defined by a simple local rule, it produces self-similar global patterns that call for an explanation. Dhar \cite{Dhar90} extended the model to any base graph and discovered the abelian property.  The abelian sandpile was independently discovered by combinatorialists \cite{BLS91}, who called it \emph{chip-firing}. Indeed, in the last two decades the subject has been enriched by an exhilarating interaction of numerous areas of mathematics, including statistical physics, combinatorics, free boundary PDE, probability, potential theory, number theory and group theory. More on this below. There are also connections to algebraic geometry \cite{Lor89,BN07,PPW13}, commutative algebra \cite{MS13,MS14} and computational complexity \cite{MN99,Babai,Cai15}. For software for experimenting with sandpiles, see \cite{Perkinson}.

Let $G=(V,E)$ be a locally finite connected graph.
A \textbf{sandpile} on $G$ is a function $s: V \to \Z$. We think of a positive value $s(x)>0$ as a number of sand grains (or ``particles'') at vertex $x$, and negative value as a hole that can be filled by particles.
Vertex $x$ is \textbf{unstable} if $s(x) \geq \deg(x)$, the number of edges incident to $x$. \textbf{Toppling} $x$ is the operation of sending $\deg(x)$ particles away from $x$, one along each incident edge.
We say that a sequence of vertices $\xx = (x_1,\ldots,x_m)$ is \textbf{legal} for $s$ if $s_i(x_i) \geq \deg(x_i)$ for all $i=1,\ldots,m$, where $s_i$ is the sandpile obtained from $s$ by toppling $x_1,\ldots,x_{i-1}$; we say that $\xx$ is \textbf{stabilizing} for $s$ if $s_m \leq \deg -1$. (All inequalities between functions are pointwise.)

\old{
\begin{lemma}
For any sandpile $s : V \to \Z$, if $\xx = (x_1,\ldots,x_m)$ is legal for $s$ and $\yy = (y_1,\ldots,y_n)$ is stabilizing for $s$, then $\xx$ is a permutation of a subsequence of $\yy$.
\end{lemma}
}

\begin{lemma}
\label{l.subseq}
Let $s : V \to \Z$ be a sandpile, and suppose there exists a sequence $\yy = (y_1,\ldots,y_n)$ that is stabilizing for $s$.
\begin{enumerate}
	\item[(i)] Any legal sequence $\xx=(x_1,\ldots,x_m)$ for $s$ is a permutation of a subsequence of $\yy$.
	\item[(ii)] There exists a legal stabilizing sequence for $s$.
	\item[(iii)] Any two legal stabilizing sequences for $s$ are permutations of each other.
\end{enumerate}
\end{lemma}

\begin{proof}
Since $\xx$ is legal for $s$ we have $s(x_1) \geq \deg(x_1)$. Since $\yy$ is stabilizing for $s$ it follows that $y_i = x_1$ for some $i$.  Toppling $x_1$ yields a new sandpile $s'$. Removing $x_1$ from $\xx$ and $y_i$ from $\yy$ yields shorter legal and stabilizing sequences for $s'$, so (i) follows by induction.

Let $\xx$ be a legal sequence of maximal length, which is finite by (i). Such $\xx$ must be stabilizing, which proves (ii).

Statement (iii) is immediate from (i).
\end{proof}

We say that $s$ \textbf{stabilizes} if there is a sequence that is
%both legal and
stabilizing for $s$.
If $s$ stabilizes, we define its \textbf{odometer} as the function on vertices
	\[ u(x) = \text{number of occurrences of $x$ in any legal stabilizing sequence for $s$}. \]
The \textbf{stabilization} $\stab{s}$ of $s$ is the result of toppling a legal stabilizing sequence for $s$. The odometer determines the stabilization, since
	\begin{equation} \label{e.stab} \stab{s} = s + \Delta u \end{equation}
where $\Delta$ is the \textbf{graph Laplacian}	
	\begin{equation} \label{e.laplacian} \Delta u(x) = \sum_{y \sim x} (u(y)-u(x)). \end{equation}
Here the sum is over vertices $y$ that are neighbors of $x$.

By Lemma~\ref{l.subseq}(iii), both the odometer $u$ and the stabilization $\hat{s}$ depend only on $s$, and not on the choice of legal stabilizing sequence, which is one reason the model is called \textbf{abelian} (another is the role played by an abelian group; see Section~\ref{s.group}).

What does a very large sandpile look like? The similarity of the two sandpiles in Figure~\ref{f.singlesource} suggests that some kind of limit exists as we take the number of particles $n \to \infty$ while ``zooming out'' so that each square of the grid has area $1/n$.  The first step toward making this rigorous is to reformulate Lemma~\ref{l.subseq} in terms of the Laplacian as follows.

\begin{namedtheorem}[Least Action Principle] \label{LAP}
If there exists $w: V \to \N$ such that
	\begin{equation} \label{e.LAP} s+ \Delta w \leq \deg -1 \end{equation}
then $s$ stabilizes, and $w \geq u$ where $u$ is the odometer of $s$. Thus,
	\begin{equation} \label{e.variational} u(x) = \inf \{ w(x) \,|\, w : V \to \N~\text{\em satisfies \eqref{e.LAP}} \}. \end{equation}
\end{namedtheorem}

\begin{proof}
If such $w$ exists, then any sequence $\yy$ such that $w(x) = \# \{i\,:\, \yy_i=x\}$ for all $x$ is stabilizing for $s$. The odometer is defined as $u(x) = \#\{i\,:\, \xx_i=x\}$ for a \emph{legal} stabilizing sequence $\xx$, so $w \geq u$ by part (i) of Lemma~\ref{l.subseq}.  The last line now follows from \eqref{e.stab}.
\end{proof}

The Least Action Principle expresses the odometer as the solution to a variational problem \eqref{e.variational}.
In the next section we will see that the same problem, without the integrality constraint on $w$, arises from a variant of the sandpile which will be easier to analyze.

\section{Relaxing Integrality: The Divisible Sandpile}

Let $\Z^d$ be the set of points with integer coordinates in $d$-dimensional Euclidean space $\R^d$, and let $\basis_1,\ldots,\basis_d$ be its standard basis vectors.  We view $\Z^d$ as a graph in which points $x$ and $y$ are adjacent if and only if $x-y = \pm \basis_i$ for some $i$. For example, when $d=1$ this graph is an infinite path, and when $d=2$ it is an infinite square grid.

In the divisible sandpile model, each point $x \in \Z^d$ has a continuous amount of mass $\sigma(x) \in \R_{\geq 0}$ instead of a discrete number of particles. Start with mass $m$ at the origin and zero elsewhere. At each time step, choose a site $x \in \Z^d$ with mass $\sigma(x)>1$ where $\sigma$ is the current configuration, and distribute the excess mass $\sigma(x)-1$ equally among the $2d$ neighbors of~$x$.  We call this a \emph{toppling}.
Suppose that these choices are sufficiently \textbf{thorough} in the sense that whenever a site attains mass $>1$, it is eventually chosen for toppling at some later time.  Then we have the following version of the abelian property.

\begin{lemma}
\label{l.stabilize}
For any initial $\sigma_0 : \Z^d \to \R$ with finite total mass, and any thorough sequence of topplings, the mass function converges pointwise to a function $\sigma_\infty : \Z^d \to \R$ satisfying $0 \leq \sigma_\infty \leq 1$.
%Moreover, $\sigma$ does not depend on the choice of thorough toppling sequence.
Any site~$z$ satisfying $\sigma_0(z)<\sigma_\infty (z)<1$ has a neighboring site~$y$ satisfying $\sigma_\infty (y)=1$.
\end{lemma}

\begin{proof}
Let $u_k(x)$ be the total amount of mass emitted from $x$ during the first $k$ topplings, and let $\sigma_k = \sigma_0 + \Delta u_k$ be the resulting mass configuration.
Since $u_k$ is increasing in $k$, we have $u_k \uparrow u_\infty$ for some $u_\infty : V \to [0,\infty]$.  To rule out the value $\infty$, consider the \emph{quadratic weight}
	\[ Q(\sigma_k) := \sum_{x \in \Z^d} (\sigma_k(x) - \sigma_0(x)) |x|^2 = \sum_{x \in \Z^d} u_k(x). \]
%Each time we topple mass $m$, the \emph{quadratic weight}
%	\[ Q(\sigma) : = \sum_{x \in \Z^d} \sigma(x) |x|^2 \]
%increases by $m$.
To see the second equality, note that $Q$ increases by $h$ every time we topple mass $h$. The set $\{\sigma_k \geq 1\}$ is connected and contains $\zero$, and has cardinality bounded by the total mass of $\sigma_0$, so it is bounded. Moreover, every site $z$ with $\sigma_k(z)>\sigma_0(z)$ has a neighbor $y$ with $\sigma_{k}(y)\geq 1$. Hence $\sup_k Q(\sigma_k) < \infty$, which shows that $u_\infty$ is bounded.

Finally, $\sigma_\infty := \lim \sigma_k = \lim (\sigma_0 + \Delta u_k) = \sigma_0 + \Delta u_\infty$. By thoroughness, for each $x \in \Z^d$ we have $\sigma_k(x) \leq 1$ for infinitely many $k$, so $\sigma_\infty \leq 1$.
\end{proof}

The picture is thus of a set of ``filled'' sites ($\sigma_\infty(z)=1$) bordered by a strip of partially filled sites ($\sigma_0(z) < \sigma_\infty(z)<1$).  Every partially filled site has a filled neighbor, so the thickness of this border strip is only one lattice spacing.  Think of pouring maple syrup over a waffle: most squares receiving syrup fill up completely and then begin spilling over into neighboring squares.  On the boundary of the region of filled squares is a strip of squares that fill up only partially (Figure~\ref{f.divsand}).

The limit $u_\infty$ is called the \textbf{odometer} of $\sigma_0$. The preceding proof did not show that $u_\infty$ and $\sigma_\infty$ are independent of the thorough toppling sequence. This is a consequence of the next result.

\begin{namedtheorem}[Least Action Principle For The Divisible Sandpile] \label{divLAP}
For any $\sigma_0 : \Z^2 \to [0,\infty)$ with finite total mass, and any $w: V \to [0,\infty)$ such that
	\begin{equation} \label{e.divLAP} \sigma + \frac{1}{2d} \Delta w \leq 1 \end{equation}
we have $w \geq u_\infty$ for any thorough toppling sequence. Thus,
	\begin{equation} \label{e.divvariational} u_\infty(x) = \inf \{ w(x) \,:\, w : V \to [0,\infty) ~\text{\em satisfies \eqref{e.divLAP}} \}. \end{equation}
\end{namedtheorem}

\begin{proof}
With the notation of the preceding proof, suppose for a contradiction that $u_k \not \leq w$ for some $k$. For the minimal such $k$, the functions $u_k$ and $u_{k-1}$ agree except at $x_k$, hence
	\[ 1 = \Bigl(\sigma + \frac{1}{2d} \Delta u_k \Bigr) (x_k) < \Bigl(\sigma + \frac{1}{2d} \Delta w \Bigr) (x_k) \leq 1\, , \]
which yields the required contradiction.
\end{proof}

\subsection{The superharmonic tablecloth}

The variational problem \eqref{e.divvariational} has an equivalent formulation:

\begin{lemma}
\label{l.twofamilies}
Let $\gamma : \Z^d \to \R$ satisfy $\frac{1}{2d} \Delta \gamma = \sigma_0 -1$. Then the odometer $u$ of \eqref{e.divvariational} is given by
	\[ u = s - \gamma \]
where
	\begin{equation} \label{e.majorant} s(x) = \inf \{ f(x) \,|\, f \geq \gamma \text{ and } \Delta f \leq 0 \}. \end{equation}
\end{lemma}

\begin{proof}
$f$ is in the set on the right side of \eqref{e.majorant} if and only if $w := f-\gamma$ is in the set on the right side of \eqref{e.divvariational}.
\end{proof}

\begin{figure}
\centering
\includegraphics[scale=.3]{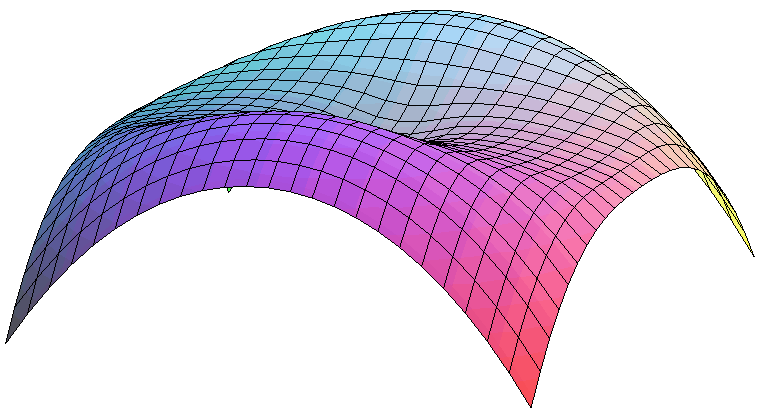} \\
\includegraphics[scale=.3]{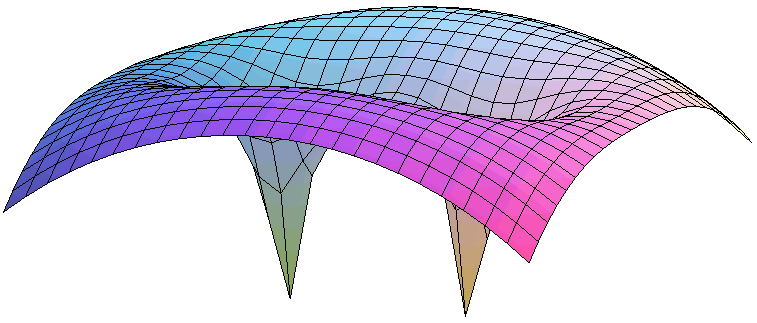}
\caption{The obstacles $\gamma$ corresponding to starting mass $1$ on each of two overlapping disks (top) and mass $100$ on each of two nonoverlapping disks.}
\end{figure}

The function $\gamma$ is sometimes called the \textbf{obstacle}, and the minimizing function $s$ in  \eqref{e.majorant} called the \textbf{solution to the obstacle problem}. To explain this terminology, imagine the graph of $\gamma$ as a fixed surface (for instance, the top of a table), and the graph of $f$ as a surface that can vary (a tablecloth). The tablecloth is constrained to stay above the table ($f \geq \gamma$) and is further constrained to be \textbf{superharmonic} ($\Delta f \leq 0$), which in particular implies that $f$ has no local minima. Depending on the shape of the table $\gamma$, these constraints may force the tablecloth to lie strictly above the table in some places.

The solution $s$ is the lowest possible position of the tablecloth. The set where strict inequality holds
	\[ D := \{x \in \Z^d \,:\, s(x)> \gamma(x) \}. \]
is called the \textbf{noncoincidence set}. In terms of the divisible sandpile, the odometer function $u$ is the gap $s-\gamma$ between tablecloth and table, and the set $\{u>0\}$ of sites that topple is the noncoincidence set.

\subsection{Building the obstacle}
\label{s.green}

%%% change the above from finite total mass to finite support, to match this section.

The reader ought now to be wondering, given a configuration $\sigma_0 : \Z^d \to [0,\infty)$ of finite total mass, what the corresponding obstacle $\gamma : \Z^d \to \R$ looks like. The only requirement on $\gamma$ is that it has a specified discrete Laplacian, namely
	\[ \frac{1}{2d} \Delta \gamma = \sigma_0 - 1. \]
Does such $\gamma$ always exist?

Given a function $f : \Z^d \to \R$ we would like to construct a function $F$ such that $\Delta F = f$.  The most straightforward method is to assign arbitrary values for $F$ on a pair of parallel hyperplanes, from which the relation $\Delta F = f$ determines the other values of $F$ uniquely.

This method suffers from the drawback that the growth rate of $F$ is hard to control.
A better method uses what is called the \textbf{Green function} or \textbf{fundamental solution} for the discrete Laplacian $\Delta$.  This is a certain function $g : \Z^d \to \R$ whose discrete Laplacian is zero except at the origin.
	\begin{equation}  \label{e.laplacianofgreen} \frac{1}{2d} \Delta g (x) = -\delta_\zero(x) = \begin{cases} -1 & x=\zero \\ 0 & x \neq \zero. \end{cases} \end{equation}
If $f$ has finite support, then we can construct $F$ as a convolution
	\[ F(x) = - f * g := - \sum_{y \in \Z^d} f(y) g(x-y) \]
in which only finitely many terms are nonzero.  (The condition that $f$ has finite support can be relaxed to fast decay of $f(x)$ as $|x| \to \infty$, but we will not pursue this.)  Then for all $x \in \Z^d$ we have
	\[ \Delta F(x) = \sum_{y \in \Z^d} f(y) \delta_\zero(x-y) = f(x) \]
as desired.  By controlling the growth rate of the Green function $g$, we can control the growth rate of $F$.  The minus sign in equation \eqref{e.laplacianofgreen} is a convention: as we will now see, with this sign convention $g$ has a natural definition in terms of random walk.

Let $\xi_1, \xi_2, \ldots$ be a sequence of independent random variables each with the uniform distribution on the set
	$  \Basis = \{\pm \basis_1, \ldots, \pm \basis_d\}. $
For $x \in \Z^d$, the sequence
	\[ X_n = \xi_1 + \ldots + \xi_n, \qquad n \geq 0 \]
is called \textbf{simple random walk} started from the origin in $\Z^d$: it is the location of a walker who has wandered from $\zero$ by taking $n$ independent random steps, choosing each of the $2d$ coordinate directions $\pm \basis_i$ with equal probability $1/2d$ at each step.

In dimensions $d \geq 3$ the simple random walk is \textbf{transient}: its expected number of returns to the origin is finite. In these dimensions we define
	\[ g(x) := \sum_{n \geq 0} \PP(X_n = x), \]
a function known as the \textbf{Green function} of $\Z^d$.  It is the expected number of visits to~$x$ by a simple random walk started at the origin in~$\Z^d$.
The identity
	\begin{equation} \label{e.dirac} -\frac{1}{2d} \Delta g = \delta_\zero \end{equation}
is proved by conditioning on the first step $X_1$ of the walk:
	\begin{align*} g(x) &= P(X_0=x) + \sum_{n \geq 1} \sum_{e \in \Basis} P(X_n=x| X_1=e) P(X_1=e). \\
				&= \delta_\zero(x) + \sum_{n \geq 1} \sum_{e \in \Basis} P(X_{n-1}=x-e) \frac{1}{2d} \end{align*}
Interchanging the order of summation, the second term on the right equals $\frac{1}{2d} \sum_{y \sim x} g(y)$, and \eqref{e.dirac} now follows by the definition of the Laplacian $\Delta$.
	
The case $d = 2$ is more delicate because the simple random walk is \textbf{recurrent}: with probability $1$ it visits $x$ infinitely often, so the sum defining $g(x)$ diverges.  In this case, $g$ is defined instead as
	\[ g(x) = \sum_{n \geq 0} \left( \PP(X_n=x) - \PP(X_n=\zero) \right). \]
 One can show that this sum converges and that the resulting function $g : \Z^2 \to \R$ satisfies \eqref{e.dirac}; see \cite{Spitzer}. The function $-g$ is called the \textbf{recurrent potential kernel} of $\Z^2$.

Convolving with the Green function enables us to construct functions on $\Z^d$ whose discrete Laplacian is any given function with finite support.
But we want more: In Lemma~\ref{l.twofamilies} we seek a function~$\gamma$ satisfying $\Delta \gamma = \sigma-1$, where $\sigma$ has finite support.  Fortunately, there is a very nice function whose discrete Laplacian is a constant function, namely the squared Euclidean norm
	\[  q(x) = |x|^2 := \sum_{i=1}^d x_i^2. \]
(In fact, we implicitly used the identity $\frac{1}{2d} \Delta q \equiv 1$ in the quadratic weight argument for Lemma~\ref{l.stabilize}.) We can therefore take as our obstacle the function
	\begin{equation} \label{eq:thediscreteobstacle} \gamma = - q - (g*\sigma). \end{equation}
In order to determine what happens when we drape a superharmonic tablecloth over this particular table $\gamma$, we should figure out what $\gamma$ looks like!  In particular, we would like to know the asymptotic order of the Green function $g(x)$ when $x$ is far from the origin.  It turns out \cite{FU,KS,Uchiyama} that
	\begin{equation} \label{e.greenasymp} g(x) =  \green(x) + O(|x|^{-d}) \end{equation}
where $\green$ is the spherically symmetric function
	\begin{equation} \label{e.greenformula} \green (x) := \begin{cases} -\frac{2}{\pi} \log |x| - a_2, & d=2; \\ a_d |x|^{2-d}, & d\geq 3. \end{cases} \end{equation}
(For $d \geq 3$ the constant $a_d = \frac{2}{(d-2)\omega_d}$ where $\omega_d$ is the volume of the unit ball in $\R^d$. The constant $a_2 = \frac{2\gamma+\log 8}{\pi}$, where $\gamma$ is Euler's constant.)
As we will now see, this estimate in combination with $-\frac{1}{2d} \Delta g = \delta_\zero$ is a powerful package. We start by analyzing the initial condition $\sigma = m\delta_\zero$ for large $m$.

\subsection{Point sources}

\begin{figure}
\centering
\includegraphics[width=.6\textwidth]{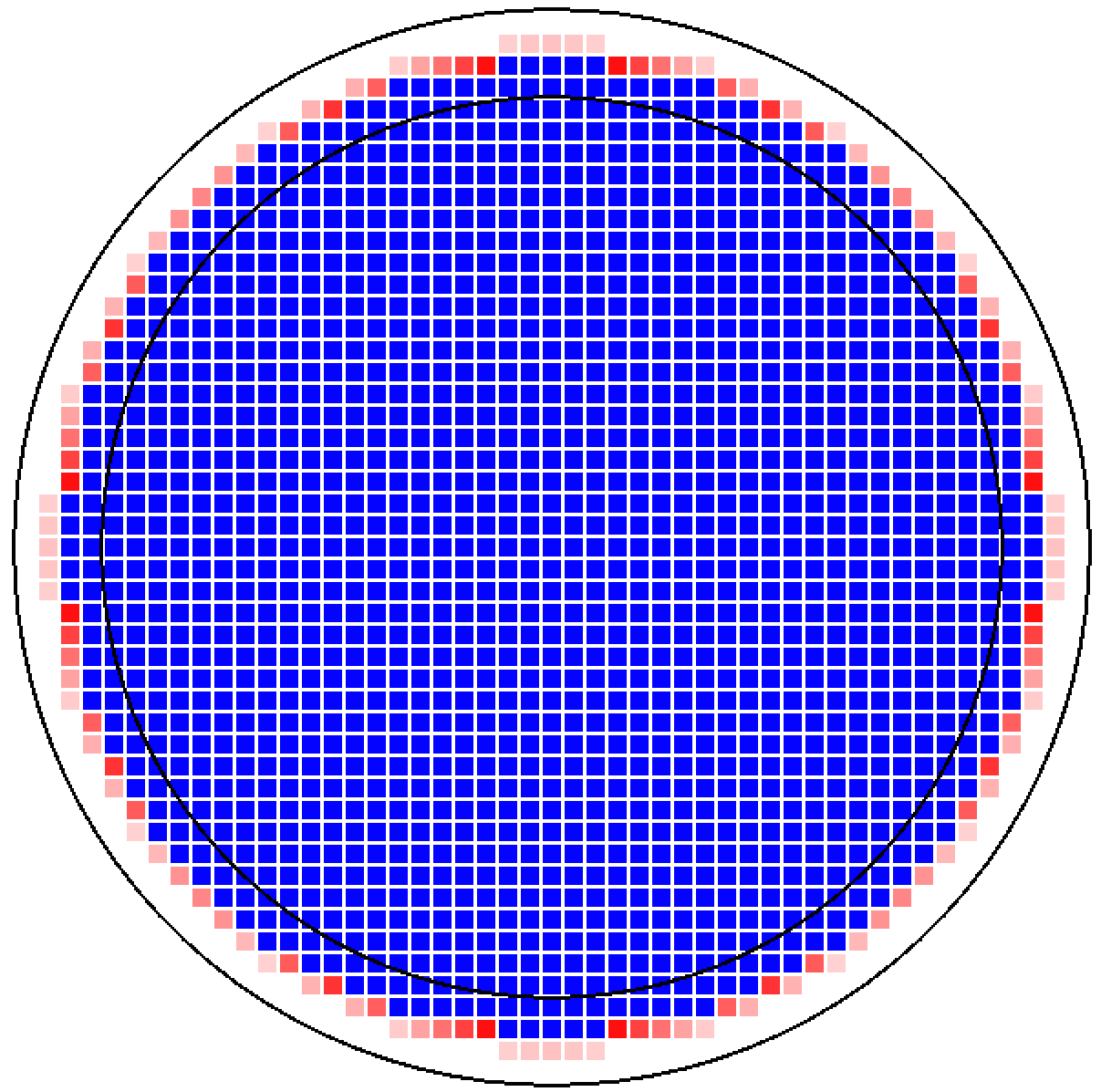}
\caption{Divisible sandpile in $\Z^2$ started from mass $m=1600$ at the origin. Each square is colored blue if it fills completely, red if it fills only partially. All red squares are contained in a thin annulus centered at the origin of radii $r \pm c$ where $\pi r^2 = m$. This is illustrated above with  $c=2$.
%, although a slightly smaller $c$ would suffice in this example.
}
\label{f.divsand}
\end{figure}

Pour $m$ grams of maple syrup into the center square of a very large waffle. Each square can hold just $1$ gram of syrup before it overflows, distributing the excess equally among the four neighboring squares. What is the shape of the resulting set of squares that fill up with syrup?
%In other words, what is the final mass configuration for the divisible sandpile started from $\sigma_0 = m\delta_\zero$?

Figure~\ref{f.divsand} suggests the answer is ``very close to a disk''. Being mathematicians, we wish to quantify ``very close'', and why stop at two-dimensional waffles? Let $B(\zero,r)$ be the Euclidean ball of radius $r$ centered at the origin in $\R^d$.

\begin{theorem} \cite{LP09} \label{t.divsand}
Let $D_m = \{\sigma_\infty =1 \}$ be the set of fully occupied sites for the divisible sandpile started from mass $m$ at the origin in $\Z^d$. There is a constant $c=c(d)$, such that
	\[ B(\zero,r-c) \cap \Z^d \subset D_m \subset B(\zero,r+c) \]
where $r$ is such that $B(\zero,r)$ has volume $m$. Moreover, the odometer $u_\infty$ satisfies
	\begin{equation} \label{e.pointsourceodomest} u_\infty(x) = mg(x) + |x|^2 - mg(r\basis_1) - r^2 +O(1) \end{equation}
for all $x \in B(\zero,r+c) \cap \Z^d$, where the constant in the $O$ depends only on $d$.
\end{theorem}

The idea of the proof is to use Lemma~\ref{l.twofamilies} to write the odometer function as
	\[ u_\infty = s-\gamma \]
for an obstacle $\gamma$ with discrete Laplacian $\frac{1}{2d} \Delta \gamma = m\delta_\zero-1$.  What does such an obstacle look like?

\old{
To build one, it suffices to find functions $g$ and $q$ with $\frac{1}{2d} \Delta g = -\delta_\zero$ and $\frac{1}{2d} \Delta q \equiv 1$, for then we can take
	\[ \gamma = -mg - q. \]
A direct calculation shows that the Euclidean norm $q(x)=|x|^2$ has $\frac{1}{2d}\Delta q \equiv 1$ (in fact, we implicitly used this identity already, in the quadratic weight argument for Lemma~\ref{l.stabilize}).
}

Recalling that the Euclidean norm $|x|^2$ and the discrete Green function $g$ have discrete Laplacians $1$ and $-\delta_\zero$, respectively, a natural choice of obstacle is
	\begin{equation} \label{e.pointsourceobstacle} \gamma(x) = -|x|^2 - mg(x). \end{equation}
The claim of \eqref{e.pointsourceodomest} is that $u(x)$ is within an additive constant of $\gamma(r\basis_1)-\gamma(x)$.  To prove this one uses two properties of $\gamma$: it is nearly spherically symmetric (because $g$ is!) and it is maximized near $|x|=r$. From these properties one deduces that $s$ is nearly a constant function, and that $\{s>\gamma\}$ is nearly the ball $B(\zero,r) \cap \Z^d$.
%%% figure here! obstacle with flat sheet over it.

The Euclidean ball as a limit shape is an example of \textbf{universality}: Although our topplings took place on the cubic lattice $\Z^d$, if we take the total mass $m \to \infty$ while zooming out so that the cubes of the lattice become infinitely small, the divisible sandpile assumes a perfectly spherical limit shape.  Figure~\ref{f.singlesource} strongly suggests that the abelian sandpile, with its indivisible grains of sand, does \emph{not} enjoy such universality. However, discrete particles are not incompatible with universality, as the next two examples show.

\section{Internal DLA}

\begin{figure}[h]
\centering
\includegraphics[height=.41\textheight]{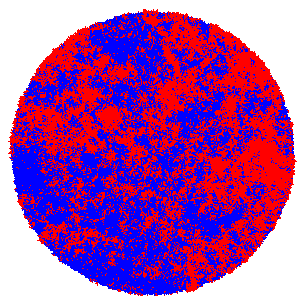}
\caption{An internal DLA cluster in $\Z^2$.  The colors indicate whether a point was added to the cluster earlier or later than expected: the random site $x(j)$ where the $j$-th particle stops is colored red if $\pi |x(j)|^2 > j$, blue otherwise.}
\end{figure}

Let $m \geq 1$ be an integer.  Starting with $m$ particles at the origin in the $d$-dimensional integer lattice $\Z^d$, let each particle in turn perform a simple random walk until reaching an unoccupied site; that is, the particle repeatedly jumps to an nearest neighbor chosen independently and uniformly at random, until it lands on a site containing no other particles.

This procedure, known as \emph{internal DLA}, was proposed by Meakin and Deutch \cite{MD86} and independently by Diaconis and Fulton \cite{DF91}. It produces a random set $I_m$ of $m$ occupied sites in $\Z^d$.  This random set is close to a ball, in the following sense.
Let $r$ be such that the Euclidean ball $B(\zero,r)$ of radius $r$ has volume $m$.  Lawler, Bramson and Griffeath~\cite{LBG92} proved that for any $\epsilon>0$, with probability~$1$ it holds that
	\[ B(\zero,(1-\epsilon)r) \cap \Z^d \subset I_m \subset B(\zero, (1+\epsilon)r)
	\qquad \mbox{for all sufficiently large $m$}. \]
A sequence of improvements followed, showing that in dimensions $d \geq 2$ the fluctuations of $I_m$ around $B(\zero,r)$ are logarithmic in $r$ \cite{Lawler95,AG1,AG2,AG3,JLS1,JLS2,JLS3}.

\section{Rotor-routing: derandomized random walk}

\begin{figure}[h]
\centering
\includegraphics[height=.6\textheight]{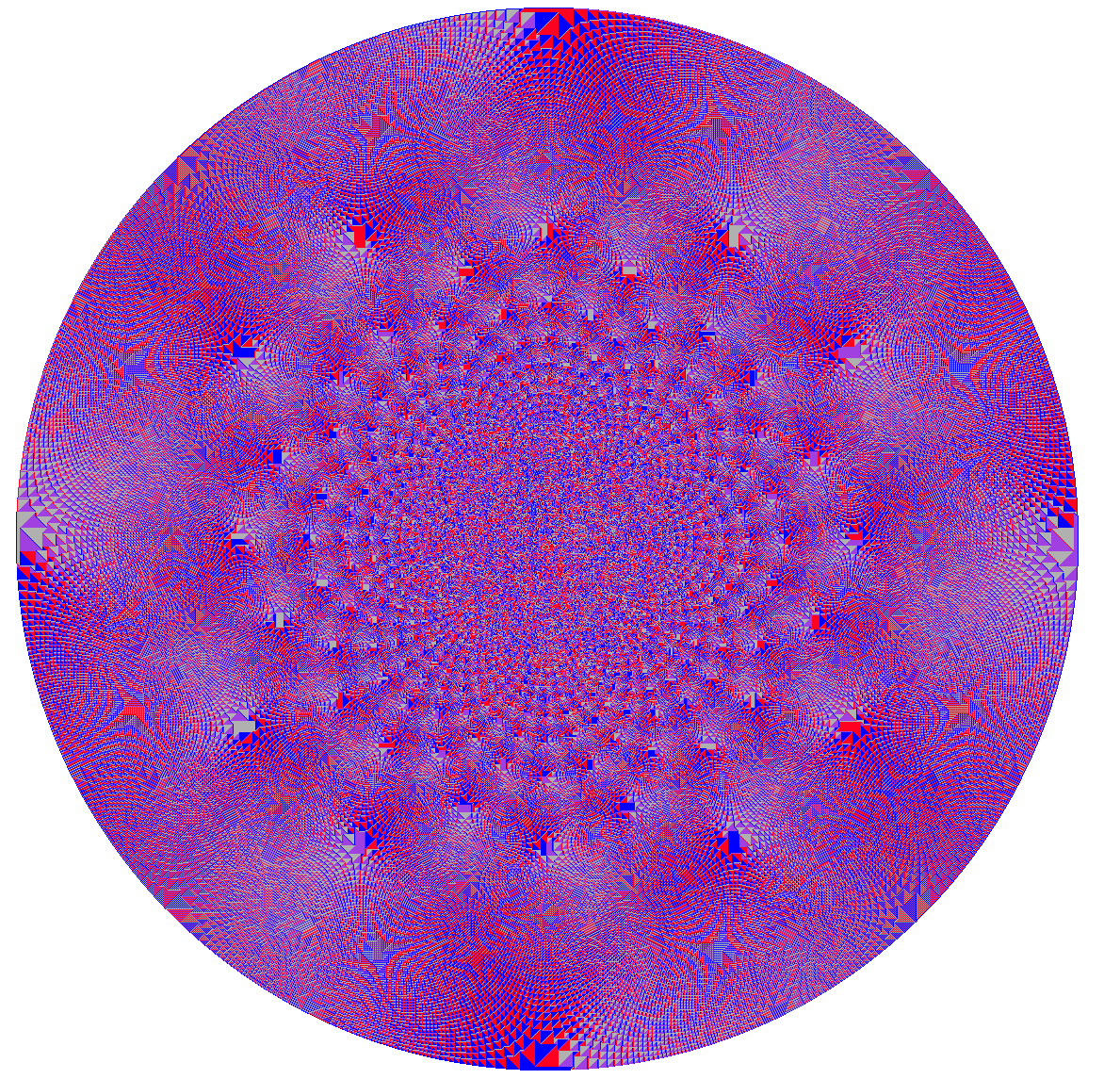}
\caption{Rotor aggregate of one million particles started at the origin in~$\Z^2$. All rotors initially pointed North and followed the \textbf{clockwise mechanism}: each rotor repeatedly cycles through the directions North, East, South, West.
Each pixel represents a site of $\Z^2$, and its color indicates the final direction of its rotor.
%(North, East, South or West).
}
\label{fig:rotor1m}
\end{figure}

In a \textbf{rotor walk} on a graph, each vertex $v$ serves its neighbors in a prescribed periodic sequence. This periodic sequence is called the \textbf{rotor mechanism} at $v$. We say that the rotor mechanism at $v$ is \textbf{simple} if each neighbor of $v$ occurs exactly once per period.  To visualize a rotor walk, label each vertex by an arrow (``rotor'') pointing toward one of its neighbors. At each time step, the walker first advances the rotor at its current location to point to the next neighbor in the periodic sequence, and then the walker moves to that neighbor.

Rotor walk has been studied in \cite{WLB96} as a model of mobile agents exploring a territory, and in \cite{PDDK96} as a model of self-organized criticality. Propp \cite{propp} proposed rotor walk as a derandomization of random walk, a perspective explored in \cite{CS06,HP10}.

In \textbf{rotor aggregation}, we start with $n$ walkers at the origin; each in turn performs rotor-router walk until it reaches a site not occupied by any other walkers.  Importantly, we do not reset the rotors between walks!  Let $R_n$ denote the resulting region of $n$ occupied sites. For example, in $\Z^2$ with the clockwise rotor mechanism whose fundamental period is North, East South West, the sequence will begin $R_1 = \{\zero\}$, $R_2 = \{\zero,\basis_1\}$, $R_3 = \{\zero,\basis_1,-\basis_2\}$.  The region $R_{10^6}$ is pictured in \figref{rotor1m}.
%The limiting shape is again a Euclidean ball.
%%% remind omega, def mechanism.

\begin{theorem}
\label{t.rotorlog}
There is a constant $C$ depending only on the dimension $d$, such that for any initial rotor configuration and any simple rotor mechanism on $\Z^d$, the rotor aggregate $R_n$ formed from $n= \floor{\omega_d r^d}$ particles started at $\zero$ satisfies
	\begin{equation} \label{e.rotorlog} B(\zero,r-C\log r) \cap \Z^d \subset R_{n} \subset B(\zero,r+C\log r). \end{equation}
\end{theorem}

The inner bound was proved in \cite{LP09}, which also included a weaker outer bound $R_{\omega_d r^d} \subset B(\zero,r+C r^{(d-1)/d} \log r)$. The rest of this section is devoted to the proof of the stronger outer bound  \eqref{e.rotorlog}, which builds on ideas of Holroyd-Propp \cite{HP10} and Jerison-Levine-Sheffield \cite{JLS1,JLS2}.

\subsection{No thin tentacles}

The first step in the proof of Theorem~\ref{t.rotorlog} is to rule out ``thin tentacles'' on the boundary of the rotor aggregate $R_n$. Namely, according to the following proposition, for any point $z_0 \in R_n$ at least a constant fraction of the lattice sites in a ball around $z_0$ must also belong to $R_n$.
%% fig: a thin tentacle.
An analogous result for internal DLA appears in \cite[Lemma A]{JLS1}.

\begin{prop}
\moniker{No Thin Tentacles}
\label{p.nothintentacles}
There is a positive constant $c$ depending only on the dimension $d$, such that for any initial rotor configuration and any simple rotor mechanism in $\Z^d$, and for any $z_0 \in R_n$ and $\rho<|z_0|$, we have
	\[ \# (B(z_0,\rho) \cap R_n) \geq c \rho^d. \]
\end{prop}

The proof of Proposition~\ref{p.nothintentacles} uses the \textbf{odometer function} $u = u_n : \Z^d \to \N$ defined by
	\[ u(x) := \text{total number of exits from $x$ by all $n$ particles during rotor aggregation}. \]
If the same particle exits $x$ several times then we include all of its exits in the count.
%Note that each point in $R_n - \{\zero\}$ has at least one neighbor $z$ with $u(z)>0$ (the occupying particle must have come from some neighbor).

We first compare the gradient of $u$ with the net number of crossings of an edge. For each directed edge $(x,y)$ of $\Z^d$ let
	\[ \theta(x,y) := N(x,y) - N(y,x) \]
where $N(x,y)$ is the total number of exits from $x$ to $y$ by all $n$ particles in rotor-router aggregation.

\begin{lemma}
\label{l.odomflow}
For all directed edges $(x,y)$ of $\Z^d$,
	\begin{equation} \label{gradodom} u(x) - u(y) = 2d\, \theta(x,y) + \beta (x,y)
\end{equation}
for a function $\beta$ on directed edges of $\Z^d$ which satisfies
	\[ |\beta(x,y)| \leq 4d-2. \]
%for all edges $(x,y)$.
\end{lemma}

\begin{proof}
Note that $u(x) = \sum_{y \sim x} N(x,y)$. There are $2d$ terms in this sum, and the definition of a simple rotor mechanism implies that any two of them differ by at at most $1$, so for $x \sim y$,
	\[ |u(x)-2d\, N(x,y)| \leq 2d-1 \, . \]
By the triangle inequality,
	\begin{align*}  | u(x)-u(y) - 2d\, \theta(x,y) | &\leq |u(x) - 2d\, N(x,y)| + |u(y) - 2d\, N(y,x)|
	\nonumber \\ &
	\leq 4d - 2. \label{eq.4d-2} \qed \end{align*}
\renewcommand{\qedsymbol}{}
\end{proof}

Next consider the discrete Laplacian of the odometer function $u$.  Since the net effect of all rotor moves is to transport $n$ particles at the origin to one particle at each site of $R_n$, we would like to say that ``$\frac{1}{2d} \Delta u = 1_{R_n} - n\delta_\zero$ on average''. (This equality holds exactly in a special case: namely, if every rotor makes an integer number of full turns, so that $N(x,y)$ depends only on $x$, then $\frac{1}{2d} \Delta u (x)$ equals the total number of entries to $x$ minus the total number of exits from $x$.) One way to make a precise statement of this form is to smooth $u$ by averaging its values over a small ball
	\[ \dball(x,k) := \oball(x,k) \cap \Z^d. \]
The proof of the next lemma is a discrete version of $\Delta = \div \nabla$.

\begin{lemma}
\label{l.smoothing}
Fix an integer $k>1$. For $f : \Z^d \to \R$, write
	\[ S_k f (x) := \frac{1}{\# \dball(x,k)} \sum_{y \in \dball(x,k)} f(y). \]
For all $x$ such that $\dball(x,k) \subset R_n -\{\zero\}$, the odometer $u=u_n$ satisfies
	\[ \frac{1}{2d} \Delta S_k u (x) = 1 + O(\frac{1}{k})  \]
where the implied constant depends only on the dimension $d$.
\end{lemma}

\begin{proof}
Note that $\Delta S_k = S_k \Delta$. Hence
	\[ \frac{1}{2d} \Delta S_k u(x)
	= \frac{1}{\# \dball(x,k)} \sum_{y \in \dball(x,k)} \frac{1}{2d} \Delta u(y). \]
Writing the Laplacian $\Delta u(y)$ as $\sum_{z \sim y} (u(z)-u(y))$, we see that the interior terms with $y,z \in Q$ cancel, leaving (by Lemma~\ref{l.odomflow})
	\[ \frac{1}{\# \dball(x,k)} \sum_{y \in \dball(x,k), \, z \notin \dball(x,k), \, z \sim y} (\theta(z,y) + \frac{1}{2d} \beta(z,y)). \]
The sum of the $\theta$ terms is the net number of particles entering $\dball(x,k)$, which is exactly $\# \dball(x,k)$: the ball starts empty and ends with exactly one particle per site, since $\dball(x,k) \subset R_n - \{\zero\}$.  Each $\beta$ term is $O(1)$ by Lemma~\ref{l.odomflow}, and there are $O(k^{d-1})$ terms in all (one for each boundary edge of $\dball(x,k)$). Dividing by $\#\dball(x,k)$ leaves $1+O(1/k)$.	
%
%Dependence on d: Each beta term is at most $4d-2$, and
% the number of terms is at most $2d^2 (2k+1)^{d-1}$ (by assuming worst case that each boundary vertex contributes d boundary edges; could easily replace this factor of d by a constant like 10, or a function C(k,d) that tends to 1 as either k or d tends to infinity.)
\end{proof}

The above proof \emph{without} any smoothing gives something much cruder. Namely, if $x \in \Z^d - \{\zero\}$ then the net number of particles entering $x$ is $\sum_y \theta(y,x) = 1_{x \in R_n}$, so
	\[ \left| \frac{1}{2d} \Delta u(x) - 1_{x \in R_n} \right| = \left| \frac{1}{2d} \sum_{y \sim x} (u(y)-u(x) - 2d\, \theta(y,x))  \right| \leq 4d-2 \]
by Lemma~\ref{l.odomflow}. In particular,
	\begin{equation} \label{e.8d^2} \left| \Delta u \right| \leq 8d^2 \quad \text{on } \Z^d-\{\zero\} \end{equation}
%But smoothing on a small scale is all we will need:
We will use Lemma~\ref{l.smoothing} for a fixed $k=k(d)$, chosen large enough so that
%the $O(1/k)$ term is smaller than $1/4$.
	\begin{equation} \label{e.taketheedgeoff}  \left| \frac{1}{2d} \Delta S_k u(x) - 1 \right| < \frac14 \end{equation}
for all $x$ such that $\dball(x,k) \subset R_n -\{\zero\}$.

Next we record the following lemma, which is proved by applying the Harnack inequality to the harmonic extensions of the functions $f(x) \pm |x|^2$ in $\dball (\zero,r)$; see the proof of Lemma 2.17 in section 7 of \cite{LP10}.

\begin{lemma}
\label{l.atmostquadratic}
%\moniker{Quadratic Growth Bound}
Fix $r>1$, and let $f$ be a nonnegative function on $\Z^d$ satisfying $f(\zero)=0$ and $|\Delta f| \leq 1$ on $\dball (\zero,2r)$. There is a constant $A_0$ depending only on $d$, such that
	\begin{align*} f(x) &\leq A_0 |x|^2  %&& \text{for all } x \in B(\zero,r) \cap \Z^d.
	\end{align*}
for all $x \in \dball (\zero,r) \cap \Z^d$.
\end{lemma}

The next lemma shows that if the odometer is large at $x$, then the cluster $R_n$ contains a ball centered at $x$.

\begin{lemma}
\label{l.smoosh}
Let $A=8d^2 A_0$.
If $u(x)> A k^2$ and $|x|>3k$, then $\dball (x,k) \subset R_n$.
\end{lemma}

\begin{proof}
We prove the contrapositive: if $\dball (x,k) \not \subset R_n$, then there is a $y \in \dball (x,k)$ such that $u(y)=0$.  Since $u \geq 0$, and $\Delta u$ is bounded by $8d^2$ on $\dball (y,2k)$ (here we use that $\zero \notin \dball (x,3k)$), it follows from Lemma~\ref{l.atmostquadratic} that $u(x) \leq Ak^2$.
\end{proof}

\begin{lemma}
\label{l.smoothsmoosh}
If $S_k u(x)>4Ak^2$ and $|x|>7k$, then $\dball (x,k) \subset R_n$.
\end{lemma}

\begin{proof}
Since $S_k u(x)$ is the average of $u$ over $\dball(x,k)$, there exists $x' \in \dball(x,k)$ with $u(x') \geq S_k u(x)$. Now   Lemma~\ref{l.smoosh} implies that $\dball(x,k) \subset \dball(x',3k) \subset R_n$.
\end{proof}

The next two lemmas follow \cite[Lemmas 4.9 and 4.10]{LP10}.
For $x \in \Z^d$ and $\rho > 0$, let	
	\[ N(x,\rho) := \# ( \dball (x,\rho) \cap R_n) \]
be the number of occupied sites within distance $\rho$ of $x$.

\begin{lemma}
\label{l.shellincrease}
Fix $x \in R_n$ and $\rho < |x|$. Let $m := \max_{\partial \dball (x,\rho)} u$. Then
	$ N(x,\rho+1) \geq (1+\frac{1}{m}) N(x,\rho). $
\end{lemma}

\begin{proof}
Note that $m \geq 1$ since $x \in R_n$. Since $\rho < |x|$, no particles start in $\dball (x,\rho)$. Each particle entering $\dball (x,\rho)$ must pass through the external vertex boundary $\partial \dball(x,\rho)$, so
	\[ N(x,\rho) \leq \sum_{y \in \partial \dball(x,\rho)} u(y). \]
The sum on the right has at most $N(x,\rho+1)-N(x,\rho)$ nonzero terms, since $\partial \dball(x,\rho) \subset \dball(x,\rho+1)-\dball(x,\rho)$. Each term is at most $m$, so
	\[ N(x,\rho) \leq m N(x,\rho+1)- mN(x,\rho). \qedhere \]
\end{proof}

Fix $k$ large enough that \eqref{e.taketheedgeoff} holds, and let
	\[ R_{n,k} = \{ x \in \Z^d \,:\, \dball (x,k) \subset R_n \}. \]

\begin{lemma}
\label{l.foam}
There are constants $n_0$ and $\rho_0$ depending only on $d$, such that for all $n>n_0$ the following holds.
For each $z_0 \in R_n$ there exists $z_1 \in R_{n,k}$ with $|z_0-z_1| < \rho_0$ and $S_k u (z_1) > 4Ak^2$. (Recall that $u=u_n$.)
\end{lemma}

\begin{proof}
Since $n$ particles start at $\zero$, we have $u(\zero) \geq n-1$. Choose $n_0$ large enough that $S_k u (\zero) > 4Ak^2$.  Take $\rho_0 := D k^{d+3}$ where the constant $D$ is chosen below.
If $|z_0| \leq \rho_0$ then take $z_1=\zero$. Otherwise, setting $m := \max_{\dball(z_0,\rho_0)} u$, we iteratively apply Lemma~\ref{l.shellincrease} to obtain
	\[ \left( 1+\frac{1}{m} \right)^{\rho_0} \leq N(z_0,\rho_0) \]
	%\geq 2^{\rho/m} \]
(here we have used that $z_0 \in R_n$ so that $N(z_0,0)=1$ and $m \geq 1$). By definition, the right side is at most $\# \dball(z_0,\rho_0)$. Taking the logarithm of both sides yields
	\[ \frac{\rho_0}{m} \leq C \log \rho_0 \]
for a constant $C$ depending only on $d$. Taking  $D$ large enough in the choice of $\rho_0$ above, it follows that $m > 4Ak^2 \# \dball (z_0,k)$, so there exists $z_1 \in \dball (z_0,\rho)$ with $S_k u(z_1)>4Ak^2$. By Lemma~\ref{l.smoosh} it follows that $z_1 \in R_{n,k}$.
\end{proof}

Lemma~\ref{l.foam} shows that near each point in $R_n$ is a point in $R_{n,k}$.
Next we show that for each point in $R_{n,k}$ there is a nearby point with large
% (smoothed)
odometer. The proof is by the maximum principle, using an idea of Caffarelli \cite{Caffarelli}.

\begin{lemma}
\label{l.caff}
If
$S_k u (z_1) > 4A k^2$,
% (which in particular implies $z_1 \in R_{n,k}$)
then for every $\rho$ satisfying $k<\rho < |z_1|-4k$, there exists $z_2 \in \dball (z_1,\rho)$ such that
	\[ S_k u(z_2) > \frac12 \rho^2. \]
\end{lemma}

\begin{proof}
Note that for every $x \in \mathcal{A} := \dball (z_1,\rho) \cap R_{n,k}$ we have $\dball (x,k) \subset R_n - \{\zero\}$,
% Here we use |z|>K.
so by \eqref{e.taketheedgeoff} the function
	\[ f(x) := S_k u(x) - \frac12 |x-z_1|^2. \]
is subharmonic in $\mathcal{A}$. By the maximum principle,
	\begin{equation} \label{e.themax} \max_{x \in \mathcal{A} \cup \partial \mathcal{A}} f(x) = \max_{x \in \partial \mathcal{A}} f(x) \end{equation}
where $\partial \mathcal{A} := \{x \in \mathcal{A}^c \,: \, x \sim y \text{ for some } y \in \mathcal{A}\}$.
Note that $z_1 \in \mathcal{A}$, and for all $x \in R_{n,k}^c \cap \partial \mathcal{A}$, Lemma~\ref{l.smoothsmoosh}
% here we use that rho < |z_1|-4k
implies that
	\[  f(x) \leq S_ku(x) \le  4Ak^2 < f(z_1) \, ,  \]
so the maximum \eqref{e.themax} must be attained at some $x \in R_{n,k} \cap \partial \mathcal{A}$. It follows that $x \notin \dball (z_1,\rho)$, so
	\[ S_k u (x) > \frac12 |x-z_1|^2 > \frac12 \rho^2. \qedhere \]
\end{proof}

\begin{proof}[Proof of Proposition~\ref{p.nothintentacles}]
%%% figure!
Fix $k=k(d)$ so that \eqref{e.taketheedgeoff} holds. By taking $c$ sufficiently small, we may assume that $\rho>\rho_0$, where $\rho_0$ was defined in Lemma~\ref{l.foam}.
% since the set on the left side contains at least 1 point, namely z_0, and the right side can be made <1 by taking c small enough.
By that lemma, there is a point $z_1\in R_{n,k}$ with $|z_1-z_0| < \rho/4$.
Now by Lemma~\ref{l.caff} there is a point $z_2$ with $S_k u(z_2) > \frac{1}{2} \rho^2$ and $|z_2-z_1| < \rho/2$.
It follows by Lemma~\ref{l.smoothsmoosh} that $R_n$ contains a ball $\dball (z_2, \rho/C)$. Since this ball has volume $c \rho^d$, and $|z_2 - z_0|< 3\rho/4$, the proof is complete.
 \end{proof}

\subsection{Holroyd-Propp bound; Probe functions}

If $h$ is a function on $\Z^d$ and $A \subset \Z^d$ is a finite set, write $h[A] := \sum_{x \in A} h(x)$.
\begin{lemma}
\label{l.HP}
\moniker{Holroyd-Propp Bound}
\cite{HP10}
For any initial rotor configuration and any simple rotor mechanism on $\Z^d$, if $h : \Z^d \to \R$ is discrete harmonic on $R_n$, then
 	\begin{equation} \label{e.HP} \left| h[R_n] - nh(\zero) \right| \leq \sum_{x \in R_n} \sum_{y \sim x} |h(x) - h(y)|. \end{equation}
\end{lemma}

\begin{proof}
For each vertex $v \in R_n$, let $v_1, \ldots, v_{2d}$ be the neighbors of $v$ in $\Z^d$ in the order they appear in $v$'s rotor mechanism. We assign a ``weight'' $w_i(v) \in \R$ to a rotor pointing from $v$ to $v_i$, so that $w_1(v)=0$ and
	\[ w_i(v) - w_{i-1}(v) = h(v) - h(v_i) \]
for each $i=1,\ldots,2d$ (taking indices modulo $2d$). These assignments are consistent since $h$ is discrete harmonic: $\sum_i (h(v) - h(v_i)) = 0$.  We also assign weight $h(v)$ to a walker located at $v$.  The sum of rotor and walker weights is unchanged by each step of rotor walk.  Initially, the sum of all walker weights is $nh(\zero)$.  After all walkers have stopped, the sum of all walker weights is $h[R_n]$.  Their difference is thus at most the change in rotor weights, which is bounded above by the right side of \eqref{e.HP}.
\end{proof}

Next we describe our choice of discrete harmonic function $h$, which is a variant of the ones used in \cite{JLS1,JLS2}. The idea is that for each point $y$ slightly outside $\dball (\zero,r)$ we can build a discrete harmonic ``probe function'' $h=h_y$ whose sum $h[R_n]$ measures how close the cluster $R_n$ comes to $y$.  A good choice of $h$ turns out to be a discrete derivative of Green's function in the radial direction: the essential properties of this $h$ are that it is
%discrete harmonic except at a few points near $y$, 
nonnegative in a neighborhood of the ball $\dball (\zero,r)$ and it decays rapidly away from $y$.

In the next lemma, $J,K,L,M$ are constants depending only on the dimension $d$.

\begin{lemma}
\label{l.probe}
\moniker{Probe Function}
Fix $r > \rho > J$.
For each $y \in \Z^d$ with $r+2\rho < |y|< r+3\rho$ there exists a function $h : \Z^d \to \R$ with the following properties.
\begin{enumerate}
\item[(i)] $h$ is discrete harmonic and nonnegative on $\dball (\zero,r+\rho)$.
\item[(ii)] $h(x) < K|x-y|^{1-d}$.
\item[(iii)] $h(x) > \frac14 \rho^{1-d}$ for all $x \in \dball (\zero,r+\rho) \cap \dball (y,2\rho)$.
\item[(iv)] $\sum_{x \in \dball (\zero, r+\rho)} \sum_{z \sim x} |h(x) - h(z)| < L \log r$.
\item[(v)] $h[\dball (\zero, r)] > (\# \dball(\zero,r)) h(\zero) - M \log r$.
\end{enumerate}
\end{lemma}

\begin{proof}
Let
	\[ h(x) := b \sum_{i=1}^d \frac{y_i}{|y|} [g(x-y+\basis_i) - g(x-y)]. \]
where $g$ is the Green function of $\Z^d$ defined in Section~\ref{s.green}, and $b$ is a small constant we will choose later.
%depending only on $d$ 

Since $g$ is discrete harmonic on $\Z^d - \{\zero\}$, the function $h$ is discrete harmonic on $\Z^d - \dball (y,1)$. To show that $h$ is nonnegative and prove (ii)-(v), we approximate $g$ by its continuum analogue $G$ of \eqref{e.greenasymp}, which yields
	\begin{equation} \label{e.capitalize}  h(x) = b \sum_{i=1}^d \frac{y_i}{|y|} [G(x-y+\basis_i) - G(x-y)] + O(|x-y|^{-d}). \end{equation}
Next observe that
	\[ \frac{\partial G}{\partial x_i} (x) = - \frac{2}{\omega_d} \frac{x_i}{|x|^d} \]
in all dimensions $d \geq 2$, and the second derivatives of $G$ are $O(|x|^{-d})$. (We remark that $G$ also has a simple physical interpretation as the Newtonian potential of a mass at the origin in $\R^d$. The gradient of $G$ is the gravitational force exerted by a mass at $\zero$ on a mass at $x$.) A first order Taylor expansion of $G$ around $x-y$ gives
	\[ h(x) = b \sum_{i=1}^d \frac{y_i}{|y|} \left[ \frac{2}{\omega_d} \frac{y_i-x_i}{|y-x|^d} \right] + O(|y-x|^{-d}) \]
where the implied constant in the error term depends only on $d$. Take $b = \omega_d/2$ so that
	\begin{equation} \label{e.innerproduct}  h(x) = \frac{ \inner{y, y-x}}{|y| |y-x|} |y-x|^{1-d} + O(|y-x|^{-d}). \end{equation}

(i) Setting $z = (r+\rho)\frac{y}{|y|}$, we have for all $x \in B(\zero,r+ \rho)$
	\begin{equation}
	\label{e.innerproductlowerbound} \inner{y, y-x} \geq  \inner{y, y-z} > \rho |y| \end{equation}
so that
	\[ h(x) > (\rho - O(1)) |y-x|^{-d} > 0 \]
on $B(\zero,r+\rho)$ (take $J$ large enough to beat the $O(1)$).
%Therefore there is a constant $J$ depending only on $d$, such that if $\rho > J$ then $h$ is nonnegative on $B(\zero,r+\rho)$.

(ii)  By Cauchy-Schwarz, the prefactor in equation \eqref{e.innerproduct} is at most $1$.
%Taking $x=\zero$ in \eqref{e.innerproductlowerbound} gives $h(\zero) = |y|^{1-d} + O(|y|^{-d})$.

(iii) For $x \in \dball (y,2\rho) \cap \dball (\zero,r+\rho)$ we have by \eqref{e.innerproduct} and \eqref{e.innerproductlowerbound}
	\[ h(x) > \frac{\rho |y|}{2 \rho |y|} \rho^{1-d} - O(\rho^{-d}).
	\]
Take $J>4$ to ensure that (iii) holds.	

(iv) This follows from \eqref{e.greenasymp}.

(v) %(We could use Theorem~\ref{t.divsand} to replace part of the following.)
Writing $H(x)$ for the main term of \eqref{e.capitalize}, we have
	\[ \int_{B(\zero,r)} h(x) \d x = \int_{B(\zero,r)} H(x) \d x + O(\log r) = \omega_d r^d H(\zero) + O(\log r). \]
Write $B^\Box$ for the union of closed unit cubes centered at the points of $\dball (\zero,r)$. Then
	\[ \sum_{x \in \dball (\zero,r)} h(s) = \int_{B^\Box} H(x) \d x + O(\log r). \]
Moreover since $h(x) < K|x-y|^{1-d}$ we have $\int_{B^\Box} H - \int_{B(\zero,r)} H = O(\log r)$.
\end{proof}

Now we prove the main result of this section. As in \cite{JLS1,JLS2} the idea is to amplify the ability of $h[R_n]$ to detect fluctuations in $R_n$: The absence of thin tentacles (Proposition~\ref{p.nothintentacles}) implies that if $R_n$ has a point within distance $\rho$ of $y$ then $R_n$ has a substantial number of points (at least $c\rho^d$) within distance $2\rho$ of $y$. Each of these points contributes substantially to the sum $h[R_n]$, but if $\rho$ is on the order of $\log r$ then $c \rho^d$ is very small compared to $n$. To show that their contribution is not swamped by the rest of $R_n$, we use the inner bound of \eqref{e.rotorlog} and the discrete mean value property Lemma~\ref{l.probe}(v).

\begin{proof}[Proof of Theorem~\ref{t.rotorlog}]
The inner bound was proved in \cite{LP09}. To prove the outer bound, let
	\[ \rho := \max(J, \max_{x \in R_n} |x| - r). \]
By \cite{LP09}
%(or No Thin Tentacles)
we have $\rho < r$ (in fact $\rho < r^{(d-1)/d} \log r$).
% actually (r log r)^{(d-1)/d} I think, which is slightly better than the sharpcirc bound.

%%%% y should be in the direction of the outradius!
We may assume that $\rho>J$. Fix $y \in \partial \dball(\zero, r+2\rho)$ such that $\dball(y, 3\rho/2)$ intersects $R_n$, and let $h$ be as in Lemma~\ref{l.probe}. By Proposition~\ref{p.nothintentacles},
	$ \# (R_n \cap B(y,2\rho)) \geq c \rho^d$.
Since $R_n \subset B(\zero,r+\rho)$ and $|y|>r+2\rho$, it follows that $R_n \cap B(y,2\rho) \subset A := B(y,2\rho) - B(y,\rho)$.  By Lemma~\ref{l.probe}(iii) we have
%$h > \frac14 \rho^{1-d}$ on $A$, hence
	\[ h [R_n \cap A] > (\frac14 \rho^{1-d})(c \rho^d) = \frac{c}{4} \rho. \]

By the discrete mean value property, Lemma~\ref{l.probe}(v), writing $n_1 = \# \dball(\zero, r-C\log r)$ we have
	\[  h [ \dball (\zero, r-C\log r) ] > n_1 h(\zero) - M \log r. \]
By Lemma~\ref{l.probe}(i) we have $h \geq 0$ on $R_n$. Now we throw away the terms $h(x)$ for $x \notin A \cup B(\zero, r-C\log r)$ and use the inner bound $\dball (\zero, r-C\log r) \subset R_n$, obtaining	
	\[ h[R_n] - n h(\zero) > h[R_n \cap A] - M \log r - (n-n_1) h(\zero). \]
The left side is at most $L \log r$, by the Holroyd-Propp Lemma~\ref{l.HP} and Lemma~\ref{l.probe}(iv). Using $n_1 > (r- 2C\log r)^d$, we get
	\[ (L+M) \log r > \frac{c}{4} \rho - (2Cd r^{d-1} \log r) h(\zero).  \]
Finally, since $h(\zero) < Kr^{1-d}$ by Lemma~\ref{l.probe}(ii), we conclude that
 	\[ \rho < C' \log r \]
as desired, with $C' = \frac{4}{c} (2CKd+L+M)$.
\end{proof}

\section{Multiple sources; Quadrature domains}

\begin{figure}
\begin{center}
\includegraphics[width=\textwidth]{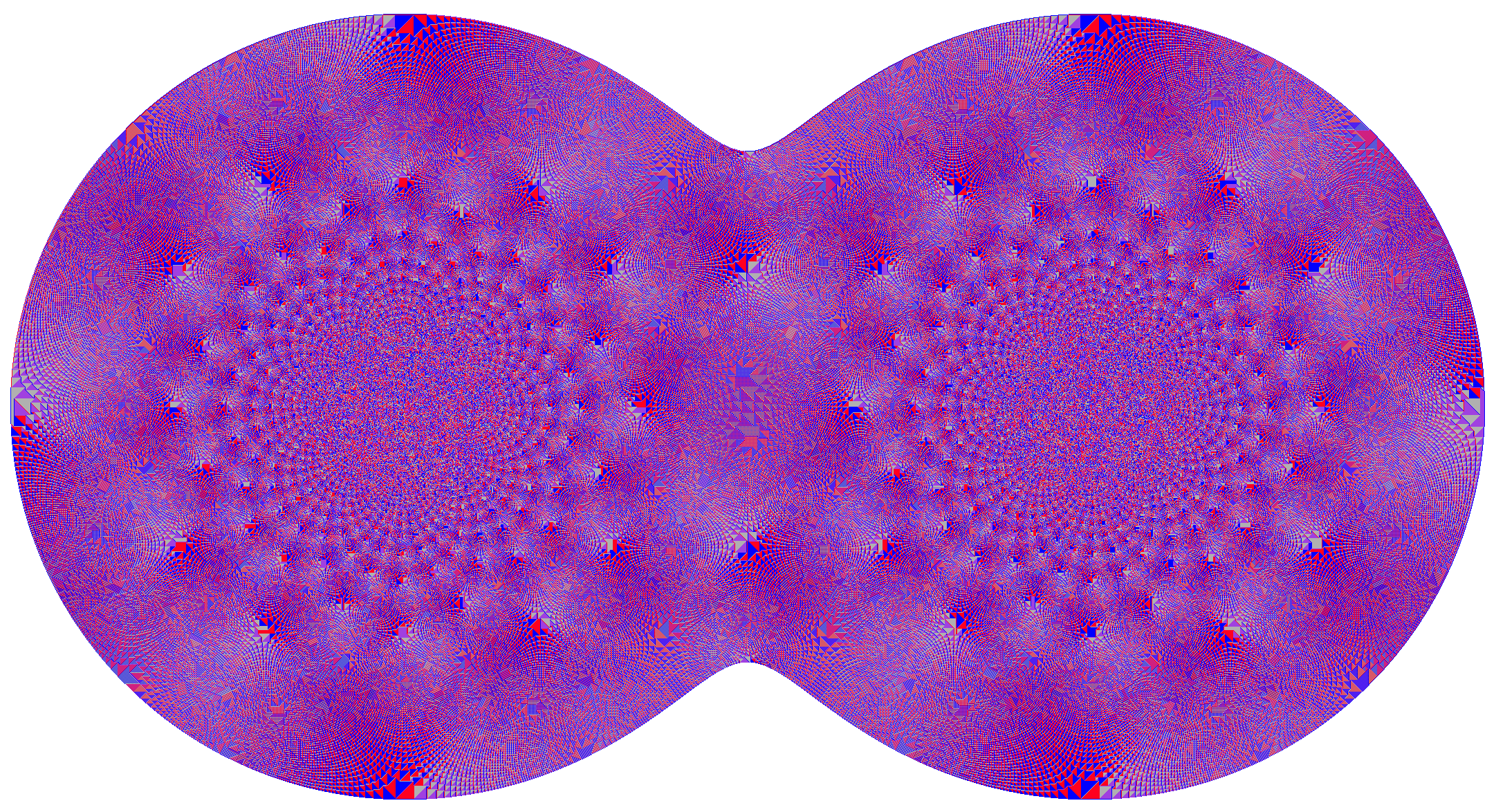}
\end{center}
\caption{Rotor-router aggregation started from two point sources in $\Z^2$.  Its scaling limit is a two-point quadrature domain in $\R^2$, satisfying \eqref{e.quad}.}
\end{figure}

The Euclidean ball as a limiting shape is not too hard to guess.  But what if the particles start at two different points of $\Z^d$?  For example, fix an integer $r \geq 1$ and a positive real number $a$, and start with $m = \floor{\omega_d (ar)^d}$ particles at each of $r\basis_1$ and $-r\basis_1$.  Alternately release a particle from $r\basis_1$ and let it perform simple random walk until it finds an unoccupied site, and then release a particle from $-r\basis_1$ and let it perform simple random walk until it finds an unoccupied site.  The result is a random set $I_{m,m}$ consisting of $2m$ occupied sites in $\Z^d$.

If $a<1$, then the distance between the source points $\pm r \basis_1$ is so large compared to the number of particles that with high probability, the particles starting at $r\basis_1$ do not interact with those starting at $-r\basis_1$.  In this case $I_{m,m}$ is a disjoint union of two ball-shaped clusters each of size~$m$.  On the other hand, if $a \gg 1$, so that the two source points are very close together relative to the number of particles released, then the cluster $I_{m,m}$ will look like a single ball of size~$2m$.  What happens in between these extreme cases?

\begin{theorem} \cite{LP10}
\label{t.twosourceintro}
For each $a>0$ there exists a deterministic domain $D \subset \R^d$ such that with probability $1$
	\begin{equation} \label{eq:twosourcedomainconvergence} \frac{1}{r} I_{m,m} \to D \end{equation}
as $r \to \infty$.
\end{theorem}

The precise meaning of the convergence of domains in \eref{twosourcedomainconvergence} is the following: given $D_r \subset \frac1r \Z^d$ and $\Omega
 \subset \R^d$, we write $D_r \to \Omega$ if for all $\epsilon>0$ we have
	\begin{equation} \label{e.ioconvintro}
	 \Omega_\epsilon \cap \frac1r \Z^d \subset D_r \subset \Omega^\epsilon
	 \end{equation}
for all sufficiently large $r$, where
	\[ \Omega_\epsilon := \{ x \in \Omega \mid \cball(x,\epsilon) \subset \Omega \} \]
and
	\[ \Omega^\epsilon := \{ x \in \R^d \mid \oball(x,\epsilon) \not\subset \Omega^c \} \]
are the inner and outer $\epsilon$-neighborhoods of~$D$.

The limiting domain $D$ is called a \textbf{quadrature domain} because it satisfies
	\begin{equation} \frac{1}{\omega_d a^d} \label{e.quad} \int_D h \d x  = h(-\basis_1) + h(\basis_1) \end{equation}
for all integrable harmonic functions $h$ on $D$, whre $\d x$ is Lebesgue measure on $\R^d$. This identity is analogous to the mean value property $\int_{B} h \d x = h(\zero)$ for integrable harmonic functions on the ball $B$ of unit volume centered at the origin.

In dimension $d=2$, the domain $D$ has a more explicit description: For $a \geq 1$ its boundary in $\R^2$ is the quartic curve
	\begin{equation} \label{eq:twosourcequartic} \left(x^2+y^2\right)^2 - 2a^2 \left(x^2 + y^2\right) - 2(x^2 - y^2) = 0. \end{equation}
When $a=1$ this factors as
	\[ (x^2 + y^2 - 2x)(x^2 + y^2 + 2x) = 0 \]
which describes the union of two unit circles centered at $\pm \basis_1$ and tangent at the origin.  This case corresponds to two clusters that just barely interact, whose interaction is small enough that we do not see it in the limit.  When $a \gg 1$, the term $2(x^2-y^2)$ is much smaller than the others, so the curve \eref{twosourcequartic} is approximately the circle
	\[ x^2 + y^2 - 2a^2 = 0. \]
This case corresponds to releasing so many particles that the effect of releasing them alternately at $\pm r e_1$ is nearly the same as releasing them all at the origin.

Theorem~\ref{t.twosourceintro} extends to the case of any $k$ point sources in $\R^d$ as follows.

\begin{theorem} \cite{LP10}
\label{multiplepointsources}
Fix $x_1, \ldots, x_k \in \R^d$ and $a_1, \ldots, a_k > 0$.  Let $x_i^\Points$ be a closest site to $x_i$ in the lattice $\frac1n \Z^d$, and let
	\begin{align*} D_n &= \{ \text{occupied sites for the divisible sandpile} \} \\
	 R_n &= \{ \text{occupied sites for rotor aggregation} \} \\
	 I_n &= \{ \text{occupied sites for internal DLA} \} \end{align*}
started in each case from $\floor{a_i n^d}$ particles at each site $x_i^\Points$ in $\frac1n \Z^d$.

Then there is a deterministic set $D \subset \R^d$ such that
	\[ D_n, R_n, I_n \to D \]
where the convergence is in the sense of \eqref{e.ioconvintro}; the convergence for $R_n$ holds for any initial setting of the rotors; and the convergence for $I_n$ is with probability $1$.
\end{theorem}

The limiting set $D$ is called a $k$-point quadrature domain. It is characterized up to measure zero by the inequalities
	\[ \int_D h \d x \leq \sum_{i=1}^k a_i h(x_i) \]
for all integrable superharmonic functions $h$ on $D$, where $\d x$ is Lebesgue measure on $\R^d$.
The subject of quadrature domains in the plane begins with Aharonov and Shapiro \cite{AS76} and was developed by Gustafsson \cite{Gustafsson83}, Sakai \cite{Sakai82,Sakai84} and others.
 The boundary of a $k$-point quadrature domain in the plane lies on an algebraic curve of degree $2k$. In dimensions $d \geq 3$, it is not known whether the boundary of $D$ is an algebraic surface!

\section{Scaling limit of the abelian sandpile on $\Z^2$}

Now that we have seen an example of a universal scaling limit, let us return to our very first example, the abelian sandpile with discrete particles.

Take as our underlying graph the square grid $\Z^2$, start with $n$ particles at the origin and stabilize. The resulting configuration of sand appears to be \emph{non-circular} (Figure~\ref{f.singlesource})---so we do not the scaling limit to be universal like the one in Theorem~\ref{multiplepointsources}.  In a breakthrough work \cite{PS13}, Pegden and Smart proved existence of its scaling limit as $n \to \infty$. To state their result, let
	\[ s_n = n\delta_\zero + \Delta u_n \]
be the sandpile formed from $n$ particles at the origin in $\Z^d$, and consider the rescaled sandpile
	\[ \bar s_n(x) = s_n(n^{1/d} x). \]

\begin{theorem} \cite{PS13} \label{t.PS}
There is a function $s : \R^d \to \R$ such that $\bar s_n \to s$ weakly-$*$ in $L^\infty(\R^d)$.
\end{theorem}

The weak-$*$ convergence of $\bar s_n$ in $L^\infty$ means that for every ball $B(x,r)$,  the average of $s_n$ over $\Z^d \cap n^{1/d} B(x,r)$ tends as $n \to \infty$ to the average of $s$ over $B(x,r)$.

The limiting sandpile $s$ is lattice dependent. Examining the proof in \cite{PS13} reveals that the lattice dependence enters in the following way.
Each real symmetric $d \times d$ matrix $A$ defines a quadratic function $q_A (\xx) = \frac12 \xx^T A \xx$ and an associated sandpile $s_A : \Z^d \to \Z$
	\[ s_A = \Delta \ceiling{q_A} \, . \]
For each matrix $A$, the sandpile $s_A$ either \textbf{stabilizes locally} (that is, every site of $\Z^d$ topples finitely often) or fails to stabilize (in which case every site topples infinitely often). The set of \textbf{allowed Hessians} $\Gamma(\Z^d)$ is defined as the closure (with respect to the Euclidean norm $\|A\|_2^2 = \mathrm{Tr}(A^T A)$) of the set of matrices $A$ such that $s_A$ stabilizes locally.

One can convert the Least Action Principle into an obstacle problem analogous to Lemma~\ref{l.twofamilies} with an additional integrality constraint. The limit of these discrete obstacle problems on $\frac1n \Z^d$ as $n \to \infty$ is the following variational problem on $\R^d$.

\begin{namedtheorem}[Limit of the least action principle]
	\begin{equation} \label{e.limitLAP} u = \inf \big\{ w \in C(\R^d) \mid w \geq -G \text{ and } D^2(w+G) \in \Gamma(\Z^d) \big\}. \end{equation}
\end{namedtheorem}

Here $G$ is the fundamental solution of the Laplacian in $\R^d$. The infimum is pointwise, and the minimizer $u$ is related to the the sandpile odometers $u_n$ by
	\[ \lim_{n \to \infty} \frac{1}{n} u_n(n^{1/2} x ) = u(x) + G(x). \]
The Hessian constraint in \eqref{e.limitLAP} is interpreted \textbf{in the sense of viscosity}:
	\[ D^2 \phi(x) \in \Gamma(\Z^d) \]
whenever $\phi$ is a $C^\infty$ function touching $w+G$ from below at $x$ (that is, $\phi(x)=w(x)+G(x)$ and $\phi-(w+G)$ has a local maximum at $x$).

\begin{figure}
\centering
\begin{tabular}{ccc}
\includegraphics[height=.25\textheight]{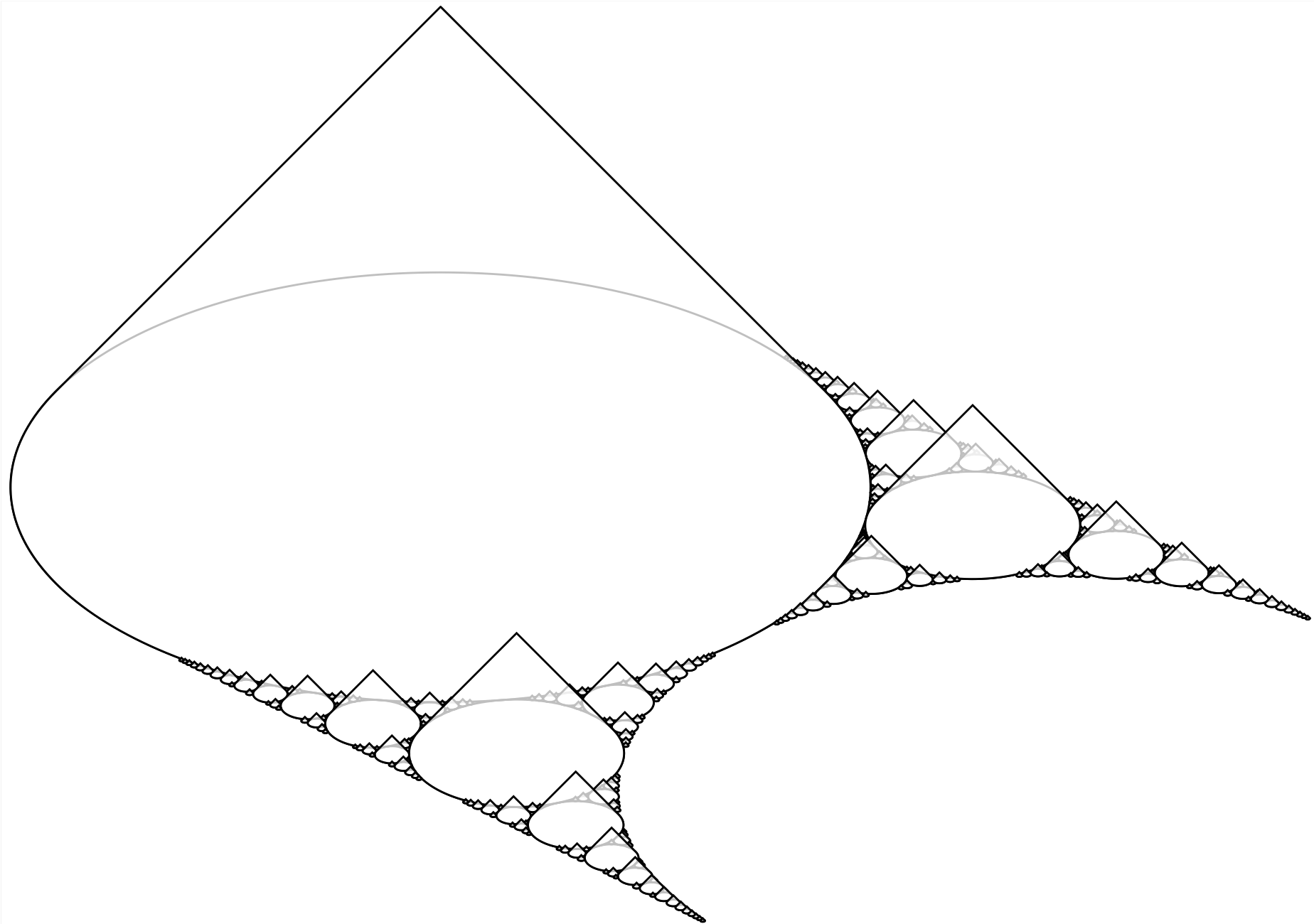}  &
\includegraphics[height=.25\textheight]{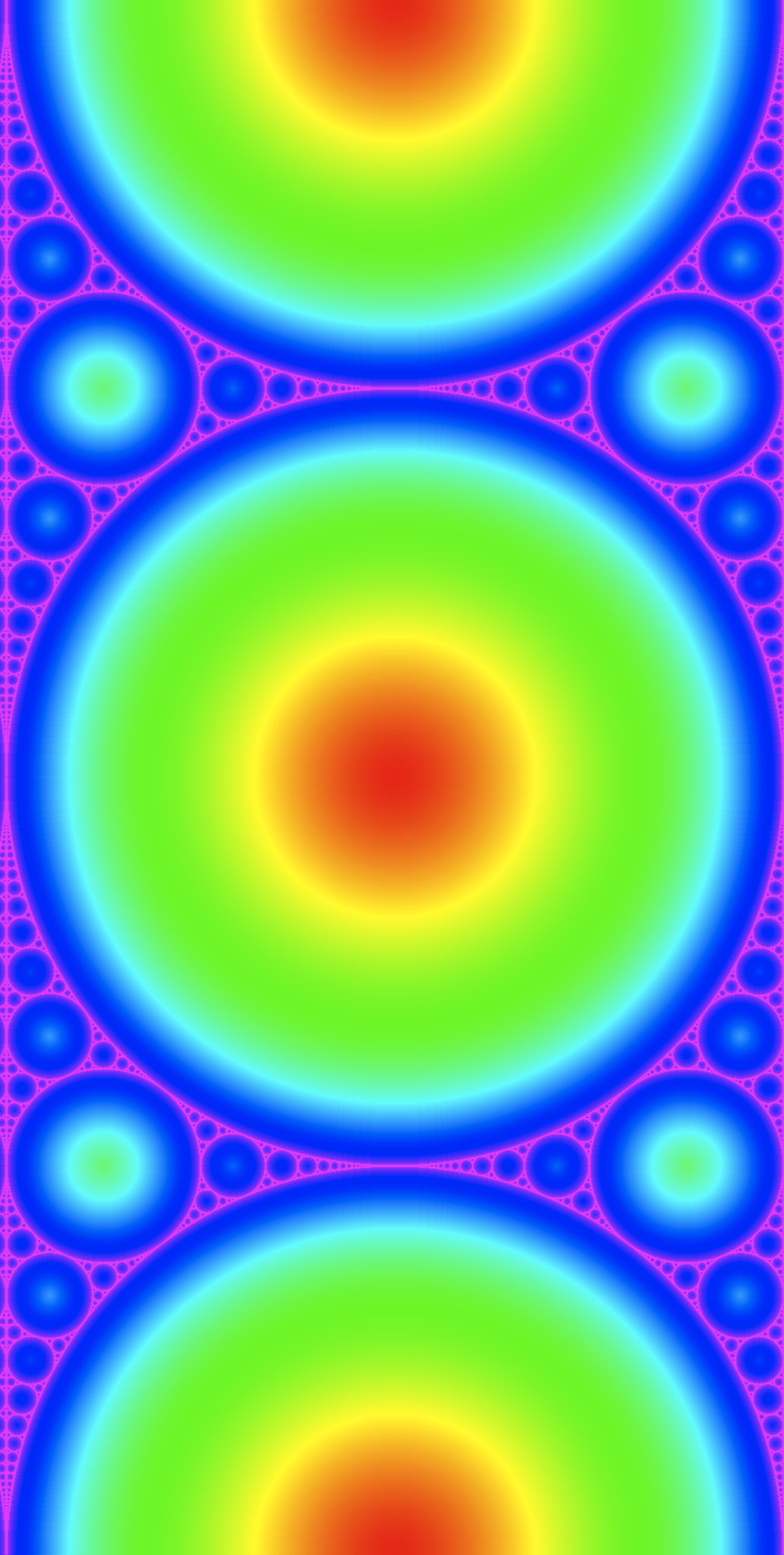} \, & \,
\includegraphics[height=.25\textheight]{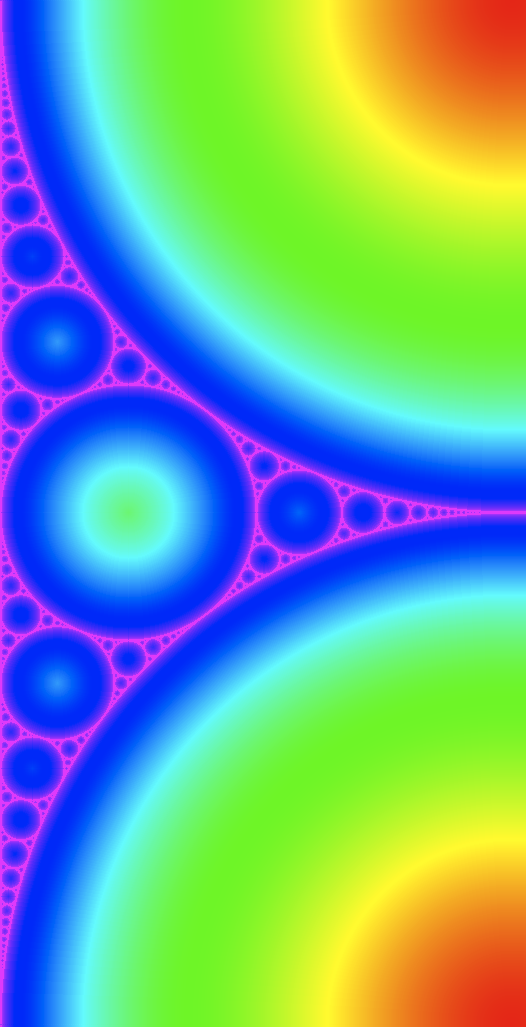} \\
(i) & (ii) & (iii)
\end{tabular}
\caption{(i) According to the main theorem of \cite{LPS2}, the set of allowed Hessians $\Gamma(\Z^2)$ is the union of slope $1$ cones based at the circles of an \textbf{Apollonian circle packing} in the plane of $2\times 2$ real symmetric matrices of trace $2$. (ii) The same set viewed from above: Color of point $(a,b)$ indicates the largest $c$ such that
%$\left[ \protect\begin{smallmatrix} c-a & b \\ b & c+a \protect\end{smallmatrix} \right] \in \Gamma(\Z^2)$.
\usebox{\smlmat}
The rectangle shown, $(a,b) \in [0,2] \times [0,4]$, extends
periodically to the entire plane.
(iii) Close-up of the lower left corner $(a,b) \in[0,1] \times [0,2]$.
}
\label{f.gammaZ2}
\end{figure}

The obstacle $G$ in \eqref{e.limitLAP} is a spherically symmetric function on $\R^d$, so the lattice-dependence arises solely from $\Gamma(\Z^d)$.  Put another way, the set $\Gamma(\Z^d)$ is a way of quantifying \emph{which features of the lattice $\Z^d$ are still detectable in the limit} of sandpiles as the lattice spacing shrinks to zero.

An explicit description of $\Gamma(\Z^2)$ appears in \cite{LPS2} (see Figure~\ref{f.gammaZ2}), and explicit fractal solutions of the \textbf{sandpile PDE}
	\[ D^2 u \in \partial \Gamma(\Z^2) \]
are constructed in \cite{LPS1}. See \cite{Pegden} for images of $\Gamma(L)$ for some other two-dimensional lattices $L$.

\section{The sandpile group of a finite graph}
\label{s.group}

Let $G=(V,E)$ be a finite connected graph and fix a \textbf{sink} vertex $z \in V$. A \textbf{stable} sandpile is now a map $s : V \setminus \{z\} \to \N$ satisfying $s(x) < \deg(x)$ for all $x \in V \setminus \{z\}$.  As before, sites $x$ with $s(x) \geq \deg(x)$ topple by sending one particle along each edge incident to $x$, but now particles falling into the sink disappear.

Define a Markov chain on the set of stable sandpiles
as follows: at each time step, add one sand grain at a vertex of $V \setminus \{z\}$ selected uniformly at random, and then perform all possible topplings until the sandpile is stable.  Recall that a state $s$ in a finite Markov chain is called \textbf{recurrent} if whenever $s'$ is reachable from $s$ then also $s$ is reachable from $s'$.
Dhar~\cite{Dhar90} observed that the operation $a_x$ of adding one particle at vertex $x$ and then stabilizing is a permutation of the set $\Rec(G,z)$ of recurrent sandpiles.  These permutations obey the relations
	\begin{align*} a_x a_y = a_y a_x \quad \text{ and } \quad
	 a_x^{\deg(x)} = \prod_{u \sim x} a_u \end{align*}
for all $x,y \in V\setminus \{z\}$. The subgroup $K(G,z)$ of the permutation group $\Sym(\Rec(G,z))$ generated by $\{a_x\}_{x \neq z}$ is called the \textbf{sandpile group} of $G$.  Although the set $\Rec(G,z)$ depends on the choice of sink vertex, the sandpile groups for different choices of sink are isomorphic (see, e.g., \cite{HLMPPW,Jar14}).

The sandpile group $K(G,z)$ has a free transitive action on $\Rec(G,z)$, so $\# K(G,z) = \# \Rec(G,z)$. One can use rotor-routing to define a free transitive action of $K(G,z)$ on the set of spanning trees of $G$ \cite{HLMPPW}. In particular, the number of spanning trees also equals $\#K(G,z)$. The most important bijection between recurrent sandpiles and spanning trees uses Dhar's burning algorithm \cite{Dhar90,MD91}.

A group operation $\oplus$ can also be defined directly on $\Rec(G,z)$, namely $s \oplus s'$ is the stabilization of $s+s'$. Then $s \mapsto \prod_{x} a_x^{s(x)}$ defines an isomorphism from $(\Rec(G,z),\oplus)$ to the sandpile group.

\begin{figure}
\centering
\includegraphics[height=0.3\textheight]{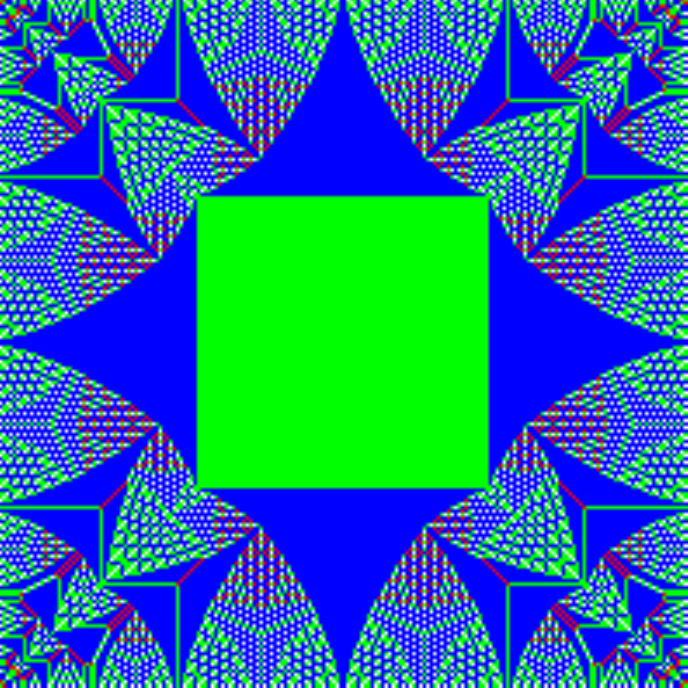} \;
\includegraphics[height=0.3\textheight]{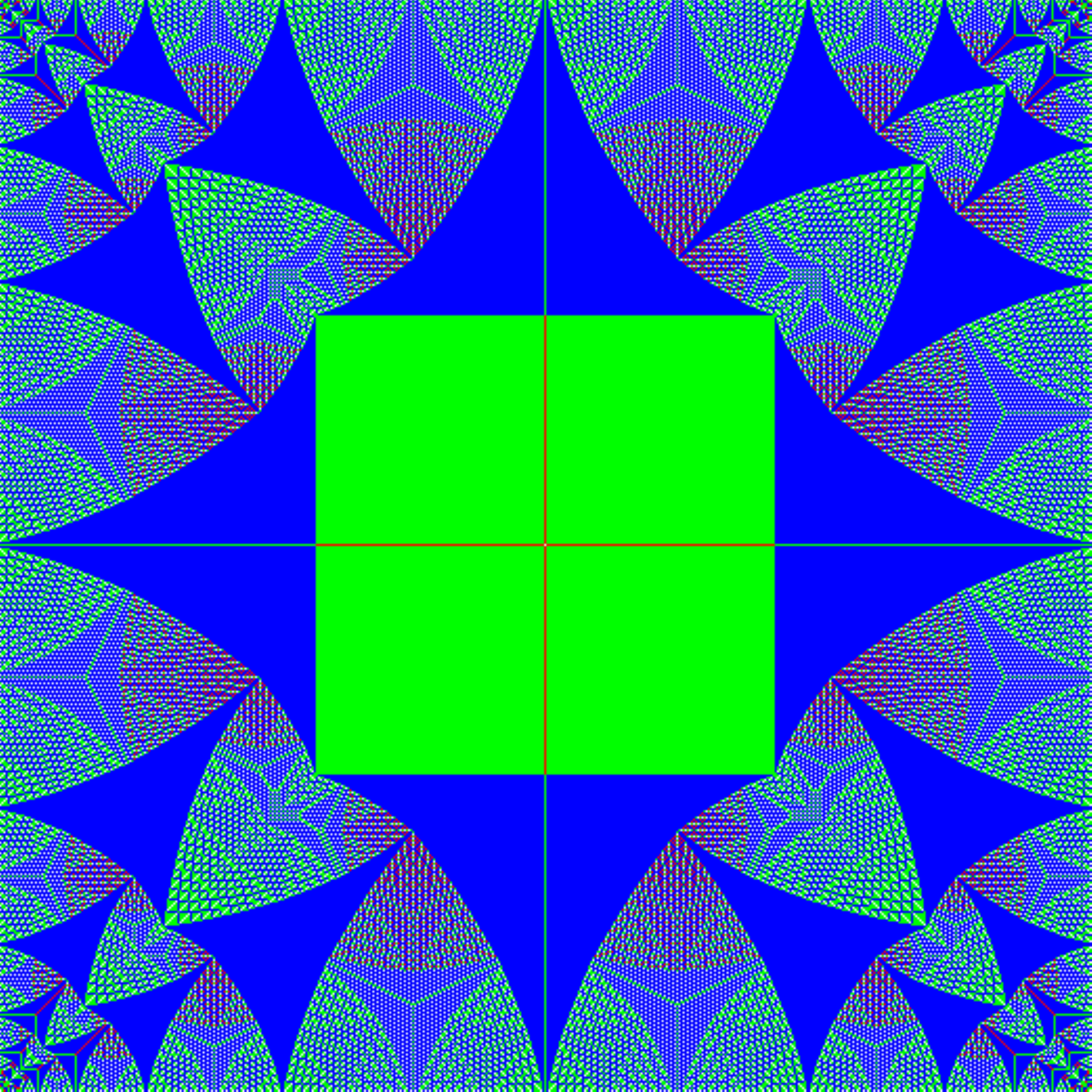}
\caption{Identity elements of the sandpile group $\Rec([0,n]^2,z)$  of the $n \times n$ grid graph with sink at the wired boundary (i.e., all boundary vertices are identified to a single vertex $z$), for $n=198$ (left) and $n=521$.}
\label{fig:identity}
\end{figure}

\section{Loop erasures, Tutte polynomial, Unicycles}

Fix an integer $d\geq 2$. The \textbf{looping constant} $\xi=\xi(\Z^d)$ is defined as the expected number of neighbors of the origin on the infinite loop-erased random walk in $\Z^d$. In dimensions $d \geq 3$, this walk can be defined by erasing cycles from the simple random walk in chronological order. In dimension $2$, one first defines the loop erasure of the simple random walk stopped on exiting the box $[-n,n]^2$ and shows that the resulting measures converge weakly \cite{Lawler80,Lawler-intersections}.

A \textbf{unicycle} is a connected graph with the same number of edges as vertices.  Such a graph has exactly one cycle (Figure~\ref{fig:unicycle}).  If $G$ is a finite (multi)graph, a \emph{spanning subgraph} of $G$ is a graph containing all of the vertices of $G$ and a subset of the edges.  A \textbf{uniform spanning unicycle} (USU) of $G$ is a spanning subgraph of $G$ which is a unicycle, selected uniformly at random.

An \textbf{exhaustion} of $\Z^d$ is a sequence $V_1 \subset V_2 \subset \cdots$ of finite subsets such that $\bigcup_{n \geq 1} V_n = \Z^d$.
Let $G_n$ be the multigraph obtained from $\Z^d$ by collapsing $V_n^c$ to a single vertex $z_n$, and removing self-loops at $z_n$.   We do not collapse edges, so $G_n$ may have edges of multiplicity greater than one incident to $z_n$.  Theorem~\ref{t.looping}, below, gives a numerical relationship between the looping constant $\xi$ and the \textbf{mean unicycle length}
	\[ \lambda_n = \EE \left[ \mbox{length of the unique cycle in a USU of $G_n$} \right]. \]
as well as the \textbf{mean sandpile height}
	\[ \zeta_n = \EE \left[\mbox{number of particles at $\zero$ in a uniformly random recurrent sandpile on $V_n$} \right]. \]

To define the last quantity of interest, recall that the \textbf{Tutte polynomial} of a finite (multi)graph $G=(V,E)$ is the two-variable polynomial
	\[ T(x,y) = \sum_{A \subset E} (x-1)^{c(A)-1} (y-1)^{c(A) + \#A - n} \]
where $c(A)$ is the number of connected components of the spanning subgraph $(V,A)$.  Let $T_n(x,y)$ be the Tutte polynomial of $G_n$. The \textbf{Tutte slope} is the ratio
	\[ \tau_n = \frac{ \frac{\partial T_n}{\partial y}(1,1)  }{(\#V_n) T_n(1,1)}. \]
A combinatorial interpretation of $\tau_n$ is the number of spanning unicycles of $G_n$ divided by the number of rooted spanning trees of $G_n$.

For a finite set $V \subset \Z^d$, write $\partial V$ for the set of sites in $V^c$ adjacent to $V$.   		

\begin{theorem} \cite{LP14}
\label{t.looping}
Let  $\{ V_n \}_{n \geq 1}$ be an exhaustion of $\Z^d$ such that
$V_1 = \{\zero \}$, $\# V_n = n$, and $\# (\partial V_n) / n \to 0$.
Let $\tau_n, \zeta_n, \lambda_n$ be the Tutte slope, sandpile mean height and mean unicycle length in $V_n$. Then the following limits exist:
	\[ \tau =  \lim_{n \to \infty} \tau_n, \qquad
	   \zeta = \lim_{n \to \infty} \zeta_n, \qquad
	   \lambda = \lim_{n \to \infty} \lambda_n. \]
Their values are given in terms of the looping constant $\xi = \xi(\Z^d)$ by
	\begin{equation} \label{eq:theformulas} \tau = \frac{\xi-1}{2}, \qquad \zeta = d + \frac{\xi-1}{2}, \qquad \lambda = \frac{2d-2}{\xi-1}. \end{equation}
\end{theorem}

Dimension two is especially intriguing, because the quantities $\xi, \tau,\zeta,\lambda$ turn out to be rational numbers.

\begin{corollary}
\label{conj:5/4}
In the case $d=2$, we have \cite{KW11,PPR11,CS12}
	\[ \xi = \frac54 \quad \mbox{and} \quad \zeta = \frac{17}{8}. \]
Hence by Theorem~\ref{t.looping},
	\[ \tau = \frac18 \quad \mbox{and} \quad \lambda = 8. \]
\end{corollary}

\begin{figure}
\centering
\includegraphics[height=0.3\textheight]{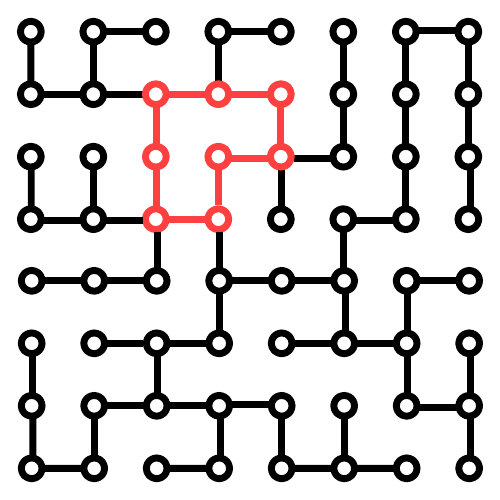}
\caption{A spanning unicycle of the $8\times 8$ square grid.  The unique cycle is shown in red.}
\label{fig:unicycle}
\end{figure}

The value $\zeta(\Z^2) = \frac{17}{8}$ was conjectured by Grassberger (see \cite{Dhar06}).  Poghosyan and Priezzhev \cite{PP11} observed the equivalence of this conjecture with $\xi(\Z^2)=\frac54$, and shortly thereafter three proofs \cite{PPR11,KW11,CS12} appeared.

The proof that $\zeta(\Z^2) = \frac{17}{8}$ by Kenyon and Wilson \cite{KW11} uses the theory of vector bundle Laplacians~\cite{Kenyon}, while the proof by Poghosyan, Priezzhev and Ruelle \cite{PPR11} uses monomer-dimer calculations.  Earlier, Jeng, Piroux and Ruelle~\cite{JPR06} had reduced the computation of $\zeta(\Z^2)$ to evaluation of a certain multiple integral which they evaluated numerically as $0.5 \pm 10^{-12}$. This integral was proved to equal $\frac12$ by Caracciolo and Sportiello \cite{CS12}, thus providing another proof.

 All three proofs involve powers of $1/\pi$ which ultimately cancel out.  For $i=0,1,2,3$ let $p_i$ be the probability that a uniform recurrent sandpile in~$\Z^2$ has exactly~$i$ grains of sand at the origin.  The proof of the distribution
\begin{align*}
p_0 &= \frac{2}{\pi^2} - \frac{4}{\pi^3} \\
p_1 &= \frac14 - \frac{1}{2\pi} - \frac{3}{\pi^2} + \frac{12}{\pi^3} \\
p_2 &= \frac38 + \frac{1}{\pi} - \frac{12}{\pi^3} \\
p_3 &= \frac38 - \frac{1}{2\pi} + \frac{1}{\pi^2} + \frac{4}{\pi^3}
\end{align*}
is completed in \cite{PPR11,KW11,CS12}, following work of \cite{MD91,Pri94,JPR06}.  In particular, $\zeta(\Z^2) = p_1 + 2p_2 + 3p_3 = \frac{17}{8}$.
%As to why $\zeta = p_1 + 2p_2 + 3p_3$ should be an exact rational number $17/8$, Jeng et al.\ write, ``The striking simplicity of this result clearly calls for a better explanation than just long calculations.''

Kassel and Wilson \cite{KW14} give a new and simpler method for computing $\zeta(\Z^2)$, relying on planar duality, which also extends to other lattices. For a survey of their approach, see \cite{Kas15}.

These objects on finite subgraphs of $\Z^d$ also have ``infinite-volume limits'' defined on $\Z^d$ itself: Lawler~\cite{Lawler80} defined the infinite loop-erased random walk, Pemantle~\cite{Pem91} defined the uniform spanning tree in~$\Z^d$, and Athreya and J\'{a}rai~\cite{AJ04} defined the infinite-volume stationary measure for sandpiles in~$\Z^d$. The latter limit uses the burning bijection of Majumdar and Dhar \cite{MD91} and the one-ended property of the trees in the uniform spanning forest \cite{Pem91,BLPS01}. As for the Tutte polynomials $T_n$ of finite subgraphs of $\Z^d$, the limit
	\[ t(x,y) := \lim_{n \to \infty} \frac1n \log T_n(x,y) \]
can be expressed in terms of the pressure of the Fortuin-Kasteleyn random cluster model.  By a theorem of Grimmett (see \cite[Theorem 4.58]{Grimmett}) this limit exists for all real $x,y>1$.  Theorem~\ref{t.looping} concerns the behavior of this limit as $(x,y) \to (1,1)$; indeed, another expression for the Tutte slope is
	\[ \tau_n =  \frac{\partial}{\partial y} \left.\left[\frac1n \log T_n(x,y) \right] \right|_{x=y=1}. \]

\section{Open problems}

We conclude by highlighting a few of the key open problems in this area.
\smallskip

\begin{enumerate}

\item Suppose $s(x)_{x \in \Z^2}$ are independent and identically distributed random variables taking values in $\{0,1,2,3,4\}$. Viewing $s$ as a sandpile, the event that every site of $\Z^2$ topples infinitely often is invariant under translation, so it has probability $0$ or $1$. We do not know of an algorithm to decide whether this probability is $0$ or $1$!  See \cite{FMR09,LMPU16}.
\medskip

\item Does the rotor-router walk in $\Z^2$ with random initial rotors (independently North, East, South, or West, each with probability $\frac14$) return to the origin with probability $1$?
The number of sites visited by such a walk in $n$ steps is predicted to be of order $n^{2/3}$ \cite{PPS98}. For a lower bound of that order, see \cite{range}. As noted there, even an upper bound of $o(n)$ would imply recurrence, which is not known!
\medskip

\item Recall that the weak-$*$ convergence in Theorem~\ref{t.PS}, proved by Pegden and Smart \cite{PS13}, means that the average height of the sandpile $s_n$ in any ball converges as the lattice spacing shrinks to zero.  A natural refinement would be to show that for any ball $B$ and any integer $j$, the fraction of sites in $B$ with $j$ particles converges.  Understanding the scaling limit of the sandpile identity elements (Figure~\ref{fig:identity}) is another appealing problem, solved in a special case by Caracciolo, Paoletti and Sportiello \cite{CPS08}.
\medskip

\item As proved in \cite{LPS2} (and illustrated in Figure~\ref{f.gammaZ2}), the maximal elements of $\Gamma(\Z^2)$ correspond to the circles in the Apollonian band packing of $\R^2$. Because the radius and the coordinates of the center of each such circle are rational numbers, each maximal element of $\Gamma(\Z^2)$ is a matrix with rational entries. Describe the maximal elements of $\Gamma(\Z^d)$ for $d \geq 3$. Are they isolated? Do they have rational entries?
\medskip

\end{enumerate}


\begin{thebibliography}{99}
\bibitem{Cai15} Hannah Cairns, Some halting problems for abelian sandpiles are undecidable in dimension three. \arxiv{1508.00161}
\end{thebibliography}
\end{document}